\documentclass[11pt, twoside, a4paper]{article}
\usepackage[utf8]{inputenc}
\usepackage{amsmath, amssymb, amsthm}            
\usepackage{amstext, amsfonts, a4}
\usepackage[english]{babel}
\usepackage{xcolor}
\usepackage{bbm}

\usepackage[margin=2.0cm]{geometry}
\usepackage{dsfont}

\usepackage[hidelinks]{hyperref}

\numberwithin{equation}{section}

\theoremstyle{plain}
	\newtheorem{theorem}{Theorem}[section]
	\newtheorem{proposition}[theorem]{Proposition}    
	\newtheorem{lemma}[theorem]{Lemma}          
	\newtheorem{corollary}[theorem]{Corollary}
	
	\newtheorem{remark}[theorem]{Remark}
	
\theoremstyle{definition}

\renewcommand{\tilde}{\widetilde}

\newcommand{\cu}{\mathbf{i}}

\newcommand{\eps}{\varepsilon}

\DeclareMathOperator{\supp}{supp}
\DeclareMathOperator{\diam}{diam}
\DeclareMathOperator{\dist}{dist}

\newcommand{\der}[2]{\frac{\mathrm{d} \hspace{0.1mm} #1}{\mathrm{d} \hspace{0.1mm} #2}}

  \def\CC{{\mathbb C}}

 \def\NN{{\mathbb N}}  
  \def\RR{{\mathbb R}}

 \def\cB{{\mathcal B}} \def\cC{{\mathcal C}} 
\def\cD{{\mathcal D}}   
\def\cG{{\mathcal G}}

\def\cP{{\mathcal P}}

\def\mP{{\mathfrak P}}

\title{Hölder regularity for collapses of point-vortices}
\author{Martin Donati\footnote{Institut Camille Jordan, Université Claude Bernard Lyon 1, 43 bd. du 11 novembre, Villeurbanne Cedex F-69622, France. Email : donati@math.univ-lyon1.fr}$\;$ and Ludovic Godard-Cadillac\footnote{Laboratoire de Mathématiques Jean Leray, CNRS UMR 6629, Université de Nantes, 2 rue de
la Houssinière, BP 92208, 44322 Nantes, France.}}
\date{\today}

\usepackage[mathlines, pagewise]{lineno}
\begin{document}

\maketitle

\begin{abstract}
The first part of this article studies the collapses of point-vortices for the Euler equation in the plane and for surface quasi-geostrophic equations in the general setting of $\alpha$ models. 
In these models the kernel of the Biot-Savart law is a power function of exponent $-\alpha$. 
It is proved that, under a standard non-degeneracy hypothesis, the trajectories of the point-vortices have a Hölder regularity up to the time of collapse. 
The Hölder exponent obtained is $1/(\alpha+1)$ and this exponent is proved to be optimal for all $\alpha$ by exhibiting an example of a $3$-vortex collapse.

The same question is then addressed for the Euler point-vortex system in smooth bounded connected domains. 
It is proved that if a given point-vortex has an accumulation point in the interior of the domain as $t\to T$, then it converges towards this point and displays the same Hölder continuity property. A partial result for point-vortices that collapse with the boundary is also established : we prove that their distance to the boundary is Hölder regular.
\end{abstract}

%%%%%%%%%%%%%%%%%%%%%%%%%%%%%%%%%%%%%%%%%%%%%%%%%%%%%%%%%%%%%%%%%%%
%%%%%%%%%%%%%%%%%%%%%%%%%%%%%%%%%%%%%%%%%%%%%%%%%%%%%%%%%%%%%%%%%%%
%%%%%%%%%%%%%%%%%%%%%%%%%%%%%%%%%%%%%%%%%%%%%%%%%%%%%%%%%%%%%%%%%%%
\section*{Introduction}
The point-vortex model is a system of ordinary differential equations introduced by Helmholtz \cite{Helmholtz_1858} that is a standard model for uniform inviscid planar fluid mechanics.
This system is derived from the two-dimensional Euler equations written in terms of vorticity.
Typically, it gives account of the natural situation where the vorticity of the fluid is sharply concentrated around some points $x_i\in\RR^2$ with $i=1\dots N$ and is (formally) replaced by Dirac masses. 
These points then evolve in the plane according to the Euler equations.
This system of differential equations has been widely studied in its different aspects, both for itself and for its links with PDE in fluid mechanics.
For an introduction to the Euler point-vortex system, we refer to~\cite[chap.\,4]{Marchioro_Pulvirenti_1993}. 
For a more extensive presentation, see~\cite{Newton_2001}.
This point-vortex model is used to describe vortex phenomena arising in different problems of fluid mechanics (see for instance: \cite{Glass_Munnier_Sueur_2018, Gryanik_Borth_Olbers_2004, Klein_Majda_Damodaran_1995, Lacave_Miot_2009, Marchioro_Pulvirenti_1984} and references therein).
The question of a rigorous derivation of the point-vortex system from the Euler equations (also called desingularization problem, or localization problem) is a standard problem that is linked to the problem of confinement and localization of vorticity~\cite{Donati_Iftimie_2020, Marchioro_Pulvirenti_1993, Smets_VanSchaftingen_2010}.

The present article focuses on the study of collapses of point-vortices at finite time $T>0$. 
The existence of collapses for the Euler point-vortex problem in the plane has been obtained independently by~\cite{Aref_1979, Grobli_1877,Novikov_1975}.
The existence of collapses of point-vortices for all bounded smooth domains has been obtained recently by~\cite{Grotto_Pappalettera_2020}.
In~\cite{Marchioro_Pulvirenti_1984}, the authors prove that when the domain is a disc, the collapses are improbable in the sense that the set of initial datum leading to a collapse in finite time has a vanishing Lebesgue measure.
This result has been extended to any bounded smooth domains (simply connected and multi-connected domains) by~\cite{Donati_2021}.
In~\cite[chap.\,4]{Marchioro_Pulvirenti_1993}, the authors obtained the improbability of collapses when the domain is the whole plane, under the non-neutral clusters hypothesis~\eqref{eq:no null partial sum} presented hereafter.
This hypothesis has been weakened by~\cite{Godard-Cadillac_2021}.
In that last article, the question of the behaviors of the point-vortices during a collapse is addressed: 
it is proved in the case of Euler equation in the plane that, under the non-neutral clusters hypothesis, the positions of the point-vortices converge as $t$ goes to the time of collapse.

In the context of geophysical fluids, a standard model is given by the surface quasi-geostrophic equations. 
These equations have many common properties with the Euler equations. 
An extension of the point-vortex theory to these equations has been proposed recently by~\cite{Geldhauser_Romito_2020}. 
In that article, the authors study the desingularization problem in order to obtain a rigorous derivation of the system (see also~\cite{Cavallaro_Garra_Marchioro_2021, Rosenzweig_2020, Godard-Cadillac_Gravejat_Smets_2020}) and the improbability of collapses of point-vortices. 
This result was improved in~\cite{Godard-Cadillac_2021}.
The collapses for the quasi-geostrophic point-vortices are studied in~\cite{Badin_Barry_2018, Godard-Cadillac_2021_b, Reinaud2021} for the $3$ vortex problem.

In~\cite{Godard-Cadillac_2021_b}, the second author of the present article studied the collapses for the $3$ vortex problem in the Euler and in the quasi-geostrophic case. 
It is proved, under the non-neutral clusters hypothesis, that in the presence of a collapse, the trajectories are Hölder continuous, with an exponent that depends explicitly on the singularity of the vorticity kernel.
Nevertheless, the proof in the previous paper fails to extend to the case of $N$ vortices.

In the present article, we deal with this problem but the approach is different and we prove that the conjectured Hölder estimate does hold with $N$ vortices. 
More precisely, we prove in Theorem \ref{thrm:plane} hereafter that for general $\alpha$-models of point-vortex dynamics, under the non neutral clusters hypothesis, the trajectories of any solution well-defined on $[0,T)$ can be extended as a  $\frac{1}{1+\alpha}$-Hölder regular trajectory on $[0,T]$. 
We also study the degenerate case where the total sum of the intensities is vanishing (Theorem~\ref{thrm:weak}).
In second time, we study the case of the Euler point-vortex dynamics in the case of a smooth bounded domain $\Omega$ and we prove in Theorem \ref{thrm:bounded} that the trajectories of the point-vortices that remain far from the boundary are $\frac{1}{2}$-Hölder regular. 
We also prove for the points collapsing with the boundary that their distance to the boundary is $\frac{1}{2}$-Hölder regular.
Our method of proof for the two theorems relies on a \emph{clusterization} by induction. We group point-vortices into clusters and show that our construction has a finite number of steps and stays stable in a large enough interval of time. 
The estimate on the interval of time relies on a precise study of the dynamics for the centers of vorticity of the clusters and a recursive argument. 

Such a mathematical result, although interesting per se, should be seen in a context of the mean-field theory for PDE. Indeed, several authors obtained results showing that the regular solutions to the 2D Euler equation or surface quasi-geostrophic equations can be approximated by $N$ point-vortices with vanishing vorticity as $N\to+\infty$. 
See~\cite{Nguyen_Rosenzweig_Serfaty_2021, Coghi_Maurelli_2020, Duerinckx_2016, Flandoli_Gubinelli_Priola_2011} and references therein. In such context, our result contributes to understanding regularity thresholds below which uniqueness for these PDE no longer holds~\cite{Grotto_Pappalettera_2020}. 
For instance we recover the expected $1/3$ Hölder regularity for the dynamics associated to the fractional laplace operator of exponent $1/2$. This exponent is known for Euler 3D in the context of the Onsager conjecture~\cite{Isett_2018}. 

%%%%%%%%%%%%%%%%%%%%%%%%%%%%%%%%%%%%%%%%%%%%%%%%%%%%%%%%%%%%%%%%%%%
%%%%%%%%%%%%%%%%%%%%%%%%%%%%%%%%%%%%%%%%%%%%%%%%%%%%%%%%%%%%%%%%%%%
%%%%%%%%%%%%%%%%%%%%%%%%%%%%%%%%%%%%%%%%%%%%%%%%%%%%%%%%%%%%%%%%%%%
\section{Presentation of the problem and main results}

%%%%%%%%%%%%%%%%%%%%%%%%%%%%%%%%%%%%%%%%%%%%%%%%%%%%%%%%%%%%%%%%%%%
%%%%%%%%%%%%%%%%%%%%%%%%%%%%%%%%%%%%%%%%%%%%%%%%%%%%%%%%%%%%%%%%%%%
\subsection{Hölder regularity for Point-vortices in the plane}

%%%%%%%%%%%%%%%%%%%%%%%%%%%%%%%%%%%%%%%%%%%%%%%%%%%%%%%%%%%%%%%%%%%
\subsubsection{Point-vortex for the Euler equation in the plane}
The motion of an homogeneous inviscid fluid in the plane is given by the two-dimensional Euler equations. 
Using the vorticity formulation, these equations take the form
\begin{equation}\tag{Eu}\label{eq:Euler}
    \left\{\begin{array}{l}\displaystyle
    \frac{\partial\omega}{\partial t}+v\cdot\nabla\omega=0,\vspace{0.2cm}\\
    v=-\nabla^\perp(-\Delta)^{-1}\,\omega.
    \end{array}\right.
\end{equation}
In the equations above, $v:[0,T]\times\RR^2\to\RR^2$ is the velocity field of the fluid and $\omega:[0,T]\times\RR^2\to\RR$ is the vorticity. 
The vorticity can be deduced from the velocity with the formula $\omega=\partial_1 v_2-\partial_2 v_1=\nabla^\perp\cdot v$, where the notation $x^\perp:=(-x_2,x_1)$ denotes the counter-clockwise rotation of angle $\pi/2$ in the plane. 
The first equation in~\eqref{eq:Euler} is a transport equation of the vorticity by the flow of the fluid which is assumed to be incompressible. 
The second equation is called the \emph{Biot-Savart law}.

To obtain the point-vortex model, we assume that the vorticity at time $t=0$ is a sum of Dirac masses $\sum_{i=1}^Na_i\,\delta_{x_i}$, where $a_i$ is the intensity of the $i^{th}$ vortex located at position $x_i$. 
The Euler equation~\eqref{eq:Euler} implies that for future times the $i^{th}$-vortex is still a Dirac mass with the same intensity $a_i$ and with a position $x_i(t)$ that evolves according to the following equations:
\begin{equation}\label{eq:evolution euler point vortex}
\der{}{t}x_i(t):=\frac{1}{2\pi}\sum_{\substack{j=1\\j\neq i}}^Na_j\frac{\big(x_i(t)-x_j(t)\big)^\perp}{\big|x_i(t)-x_j(t)\big|^2}.
\end{equation}
These equations of evolution are obtained from the Euler equations~\eqref{eq:Euler} by computing the velocity of the fluid in the case where the vorticity is a sum of Dirac masses $\sum_{i=1}^Na_i\,\delta_{x_i(t)}$.
This gives
\begin{equation}
    v(t,x)=\frac{1}{2\pi}\sum_{j=1}^Na_j\frac{\big(x-x_j(t)\big)^\perp}{\big|x-x_j(t)\big|^2},
\end{equation}
where we used that the Green function of the $(-\Delta)$ operator in the plane is given by:
\begin{equation}x\longmapsto\frac{1}{2\pi}\,\ln\bigg(\frac{1}{|x|}\bigg).\end{equation}
Under the natural assumption that a given point-vortex does not interact with itself at distance $0$, the transport equation for the vorticity gives~\eqref{eq:evolution euler point vortex}. 
For the sake of generalization, we define the function $G_\alpha:\RR^2\setminus\{0\}\to\RR$, with $\alpha\geq0$ by
\begin{equation}
\left\{ \begin{array}{cc}
     G_1(x) = \ln(|x|)  &  \vspace{2mm} \\
     \displaystyle G_\alpha(x) = \frac{|x|^{1-\alpha}}{1-\alpha} & \text{ if } \alpha \neq 1.
\end{array}\right.
\end{equation}
In particular, it satisfies for every $\alpha \ge 0$ that
\begin{equation}\label{eq:K alpha}
|\nabla G_\alpha(x)|=\frac{1}{|x|^\alpha}.\end{equation}
The function $G_\alpha$ for a fixed value of $\alpha\geq0$ will be referred throughout this article as \emph{the kernel profiles of exponent} $\alpha$. 
The Euler point-vortex equation can then be written using the kernel profile $G_\alpha$ with $\alpha=1$:
\begin{equation}\label{eq:evolution euler K1}
    \der{}{t}x_i(t):=\frac{1}{2\pi}\sum_{\substack{j=1\\j\neq i}}^Na_j\nabla^\perp G_1\big(x_i(t)-x_j(t)\big).
\end{equation}

%%%%%%%%%%%%%%%%%%%%%%%%%%%%%%%%%%%%%%%%%%%%%%%%%%%%%%%%%%%%%%%%%%%
\subsubsection{Point-vortex for the quasi-geostrophic equations}
In geophysical fluid dynamics, a standard model is given by the surface quasi-geostrophic equations :
\begin{equation}\tag{SQG}\label{eq:SQG}
    \left\{\begin{array}{l}\displaystyle
    \der{\omega}{t}+v\cdot\nabla\omega=0,\vspace{0.2cm}\\
    v=-\nabla^\perp(-\Delta)^{-s}\,\omega,
    \end{array}\right.
\end{equation}
with $0<s<1$. 
These equations give account of a quasi-stratified fluid subject to Brünt-Väisälä oscillations evolving in a rapidly rotating frame. 
This model is widely used for weather forecast~\cite{Pedlowsky_1987, Vallis_2006}.
The transport equation of the vorticity in~\eqref{eq:SQG} is the same as in~\eqref{eq:Euler}. 
The difference lays in the Biot-Savart law that involves a fractional Laplace operator in this case. 
Formally, if we take $s=1$ in~\eqref{eq:SQG} then it gives back~\eqref{eq:Euler}.
In the plane, the Green function of the fractional Laplace operator is given by
\begin{equation}
    x\longmapsto\frac{C_s}{|x|^{2(1-s)}},\qquad\text{where}\qquad C_s:=\frac{\Gamma(1-s)}{2^{2s}\,\pi\,\Gamma(s)},
\end{equation}
with $\Gamma$ the standard Gamma function.

In~\cite{Geldhauser_Romito_2020}, the authors suggested to exploit the proximity of formulations between~\eqref{eq:Euler} and~\eqref{eq:SQG} to derive a more general point-vortex model. 
Indeed we formally choose an initial value for $\omega$ in~\eqref{eq:SQG} to be a sum of Dirac masses $\sum_{i=1}^Na_i\,\delta_{x_i}$. 
A reasoning similar to the case of the Euler equation (but involving this time the fractional Green function) gives that the structure of Dirac masses persists and the position of the Dirac masses evolves in time according to:
\begin{equation}\label{eq:evolution alpha models 2}
      \der{}{t}x_i(t)=2C_s(1-s)\sum_{\substack{j=1\\j\neq i}}^Na_j\nabla^\perp G_{(3-2s)}\big(x_i(t)-x_j(t)\big).
\end{equation}
These formulations involving the kernel profiles $G_\alpha$ allows us to consider the general $\alpha$-point-vortex model in the plane for $\alpha\geq 0$: 
\begin{equation}\label{eq:evo alpha}
      \der{}{t}x_i(t):=\sum_{\substack{j=1\\j\neq i}}^Na_j\,\nabla^\perp G_\alpha\big(x_i(t)-x_j(t)\big)=\sum_{\substack{j=1\\j\neq i}}^Na_j\,\frac{(x_i(t)-x_j(t))^\perp}{|x_i(t)-x_j(t)|^{\alpha+1}}.
\end{equation}
Up to a time rescaling, the point-vortex model for the  Euler equations in the plane~\eqref{eq:evolution euler K1} or quasi-geostrophic point-vortex model~\eqref{eq:evolution alpha models 2} are particular cases of the more general evolution system~\eqref{eq:evo alpha} that we study in this article.

The point-vortex model~\eqref{eq:evo alpha} admits a unique solution by the Cauchy-Lipschitz theorem provided that the right-hand side of~\eqref{eq:evo alpha} remains bounded and Lipschitz.
This means that this system is well posed in the absence of collapses of point-vortices. In other words, if $T>0$ the maximal time of existence of the solution is finite then
\begin{equation}\label{def:collapse}
    \liminf_{t\to T^-}\;\min_{i\neq j}\,|x_i(t)-x_j(t)|=0.
\end{equation}
The study of collapse situations~\eqref{def:collapse} has several interests. First, understanding collapses permits to avoid them and ensures the system to be well-defined for all times. Moreover, such a study is necessary to understand how to extend the trajectories after a collapse in a way that is physically relevant. These problems are linked to the issue of topological regularization of trajectories~\cite{Hiraoka_2008, Hiraoka_2009} and the dissipation of physical quantities such as energy or enstrophy~\cite{Gotoda_Sakajo_2016, Gotoda_Sakajo_2017,Sakajo_2012}. We also expect to better understand the losses of uniqueness for the underlying PDE (see~\cite{Grotto_Pappalettera_2020} for instance).

%%%%%%%%%%%%%%%%%%%%%%%%%%%%%%%%%%%%%%%%%%%%%%%%%%%%%%%%%%%%%%%%%%%
\subsubsection{Main result of the article}
The main result of this article is the following:
\begin{theorem}[Hölder regularity for $\alpha$-point-vortex dynamics in the plane]\label{thrm:plane}
Consider the $\alpha$-point-vortex dynamic \eqref{eq:evo alpha} for a given $\alpha\geq0$ with intensities $a_i\neq 0$. 
Consider an initial datum $X\in\RR^{2N}$ such that the associated dynamic of point-vortices is well defined on $[0,T)$ for some final time $T>0$ finite. We assume that the intensities satisfy the ``\emph{non neutral clusters hypothesis}'':
\begin{equation}\label{eq:no null partial sum}
\forall\;P\subseteq\{1,\dots, N\}\; s.t.\;\;P\neq\emptyset,\qquad\sum_{i\in P}a_i\neq0.
\end{equation}
Then the trajectories of the point-vortices are Hölder continuous for all $i\in\{1,\dots, N\}$:
\begin{equation}
    \forall\,t_1<t_2\in[0,T),\qquad|t_2-t_1|\leq1\quad\Longrightarrow\quad\big|x_i(t_2)\,-\,x_i(t_1)\big|\:\leq\:C\,|t_2-t_1|^\frac{1}{\alpha+1},
\end{equation}
where the constant $C$ depends only on $N$, $\alpha$ and on the intensities $a_i$. In particular, the trajectories converge as $t \to T^-$.
\end{theorem}
The $1/2$-Hölder regularity obtained is the same as obtained for the collapse of $3$ vortices by~\cite{Aref_2010} when $\alpha=1$ (Euler case). 
When $\alpha=2$ (SQG case with exponent $s=1/2$) we also find the $1/3$-Hölder regularity known for the $3$-vortex problem~\cite{Reinaud2021}. 
As stated in introduction, this exponent is the same as the critical exponent of the Onsager conjecture. 
This is due to the fact that the Green function for $(-\Delta)^\frac{1}{2}$ in the plane is the same as the one for $(-\Delta)$ for the whole three dimensional space.
We also note that the Hölder exponent $1/(\alpha+1)$ obtained in the general case is the same as the one previously obtained in the partial result~\cite{Godard-Cadillac_2021_b}. 
We prove in Appendix~\ref{appendix:optimality} the existence of self-similar collapses for all $\alpha$ and we observe that the regularity at the time of collapse is $\frac1{\alpha+1}$-Hölder and not better. This shows the optimality of the Hölder regularity obtained in the above theorem.

The non-neutral clusters hypothesis is a non-degeneracy hypothesis that is standard in a context of collapses of point-vortices~\cite{Marchioro_Pulvirenti_1984,Marchioro_Pulvirenti_1993,Geldhauser_Romito_2020,Godard-Cadillac_2021}. 
Nevertheless, one can wonder what happens if this hypothesis is weakened.
If we allow the total sum of vortices to be $0$, then we are only able to prove a much weaker version of Theorem~\ref{thrm:plane}:

\begin{theorem}\label{thrm:weak}
Same hypothesis as Theorem~\ref{thrm:plane} except~\eqref{eq:no null partial sum} is replaced by the weaker ``\emph{non neutral \emph{sub}-clusters hypothesis}'': 
\begin{equation}\label{eq:no null sub partial sum}
\forall\;P\subseteq\{1,\dots, N\}\; s.t.\;\;P\neq\emptyset\;\;\text{and}\;\;\{1,\dots, N\},\qquad\quad\sum_{i\in P}a_i\neq0.
\end{equation}
Then we have for all indices $i\neq j\in\{1,\dots, N\}$ such that $ \liminf_{t \to T^-} |x_i(t)-x_j(t)|=0$:
\begin{equation}
    \forall\,t\in[0,T),\qquad|T-t|\leq1\quad\Longrightarrow\quad\big|x_i(t)-x_j(t)\big|\:\leq\:C\,|T-t|^\frac{1}{\alpha+1},
\end{equation}
where the constant $C$ depends only on $N$, $\alpha,$ and on the intensities $a_i$.
\end{theorem}
In Theorem \ref{thrm:weak}, we observe that assuming~\eqref{eq:no null sub partial sum} is enough to prove the convergence of the \emph{relative} trajectories (ie : the map $t\mapsto x_i(t) - x_j(t)$, for $i \neq j$ is a continuous map on $[0,T]$). Nevertheless, we cannot obtain the convergence of the actual trajectories $t\mapsto x_i(t)$ nor even the Hölder regularity of the relative trajectories.

%%%%%%%%%%%%%%%%%%%%%%%%%%%%%%%%%%%%%%%%%%%%%%%%%%%%%%%%%%%%%%%%%%%
%%%%%%%%%%%%%%%%%%%%%%%%%%%%%%%%%%%%%%%%%%%%%%%%%%%%%%%%%%%%%%%%%%%
\subsection{Hölder regularity for Point-vortices in smooth bounded domains}

%%%%%%%%%%%%%%%%%%%%%%%%%%%%%%%%%%%%%%%%%%%%%%%%%%%%%%%%%%%%%%%%%%%
\subsubsection{Point-vortex for the Euler equation in smooth bounded domains}
In the case of the Euler equations in a bounded simply-connected smooth connected domain of $\RR^2$ with impermeability condition at the boundary, it is possible to proceed to an analogous construction (see~\cite{Donati_2021}) and get
\begin{equation}
    \der{}{t}x_i(t):=a_i\nabla^\perp_x\gamma_\Omega\big(x_i(t),x_i(t)\big)-\sum_{\substack{j=1\\j\neq i}}^Na_j\nabla^\perp_x\,\cG_\Omega\big(x_i(t),x_j(t)\big),
\end{equation}
where $\cG_\Omega$ is the Green function of the $(-\Delta)$ operator in $\Omega$ and where 
\begin{equation}\label{def:robin function} \gamma_\Omega(x,y):=-\cG_\Omega(x,y)-\frac{1}{2\pi}G_1\big(x-y\big)\end{equation} 
is harmonic in both variables. 
This function $\gamma_\Omega$ satisfies also the following estimate:
\begin{equation}\label{eq:kernel robin estimate}
    \forall\;x,y\in\Omega,\quad\big|\nabla_x\,\gamma_\Omega(x,y)\big|\,\leq\,\frac{C_\Omega}{\dist\big(x,\partial\Omega\big)}
\end{equation}
where $C_\Omega$ is a constant that depends only on $\Omega$ and where $\dist(x,A):=\inf_{y\in A}\,|x-y|.$
Indeed, we know (see for instance~\cite{Donati_2021}) that $\gamma_\Omega$ satisfies
\begin{equation}\label{eq:kernel robin general estimate}
     \forall\;x,y\in\Omega,\quad\big|\nabla_x\,\gamma_\Omega(x,y)\big|\,\leq\,\frac{C_\Omega}{|x-y|}
\end{equation}
and since the map $y\mapsto \nabla_x\,\gamma_\Omega(x,y)$ is harmonic, then~\eqref{eq:kernel robin estimate} follows by application of the maximum principle.
%\begin{equation}
%    |\nabla_x\,\gamma_\Omega(x,y)| \le \max_{z\in\partial\Omega}|  \nabla_x\,\gamma_\Omega(x,z)| \le \max_{z\in\partial\Omega}\frac{C_\Omega}{|x-z|} = \frac{C_\Omega}{\dist(x,\partial\Omega)}.
%\end{equation}

In the case where the domain is non simply connected, an additional contribution must be taken into account coming from the circulation of the fluid around the holes of the domain. 
The $M$ holes in the domain $\Omega$ are noted $\Omega_m$ and the domain delimited by the exterior boundary, namely $\Omega$ ``without holes'' is denoted by $\Omega_0$. 
For any of these holes, the total circulation of the flow at the surface $\partial\Omega_m$ is preserved by the motion and is denoted by $\xi_m$. 
The equations for the point-vortex system in a non simply connected domain is given by (see for instance the Biot-Savart law given by \cite[Proposition 2.9]{IftimieLopesLopesWeakBounded2020}):
\begin{equation}\label{eq:evolution euler non simply connected}
        \der{}{t}x_i(t):=a_i\nabla^\perp_x\gamma_\Omega\big(x_i(t),x_i(t)\big)-\sum_{\substack{j=1\\j\neq i}}^Na_j\nabla^\perp_x\,\cG_\Omega\big(x_i(t),x_j(t)\big)+\sum_{m=1}^Mc_m(t)\beta_m\big(x_i(t)\big).
\end{equation}
In the equations above, the functions $\beta_m$ for $m \in \{1,\ldots,M\}$ are the harmonic vector fields, defined by
\begin{equation}\label{def:beta_m}
    \begin{cases}
    \nabla\cdot \beta_m = 0 & \text{ in }\Omega \vspace{2mm} \\
    \mathrm{curl} \, \beta_m = 0 & \text{ in }\Omega \vspace{2mm} \\
    \beta_m \cdot n = 0 & \text{ on } \partial\Omega \vspace{2mm}\\
    \displaystyle \int_{\Gamma_\ell} \beta_m \cdot (-n^\perp) \mathrm{d} s= \delta_{m,\ell} & \forall \ell \in \{1,\ldots,M\},
    \end{cases}
\end{equation}
where $\delta_{m,\ell}$ represents here the Kronecker symbol and $n$ is the exterior normal vector.
The quantity $c_m$ is given by:
\begin{equation}\label{def:c_m}
    c_m(t)=\xi_m+\sum_{k=1}^Na_k\,w_m\big(x_k(t)\big),
\end{equation}
where the $w_m$ are the harmonic maps defined by
\begin{equation}
    \left\{\begin{array}{ll}
-\Delta w_m=0&\qquad\text{in }\Omega,\\
w_m=1&\qquad\text{on }\partial\Omega_m,\\
w_m=0&\qquad\text{on }\partial\Omega\setminus\partial\Omega_m.
    \end{array}\right.
\end{equation}

Note that the extra term appearing in the case of non simply connected domains is a bounded term, by standard elliptic estimates since $\partial\Omega$ is smooth. Using the relation~\eqref{def:robin function}, the evolution equation~\eqref{eq:evolution euler non simply connected} can be rewritten under the following expanded form:
\begin{equation}\label{eq:evolution euler reformulate}
            \der{}{t}x_i(t)=\sum_{j=1}^Na_j\nabla^\perp_x\gamma_\Omega\big(x_i(t),x_j(t)\big)+\frac{1}{2\pi}\sum_{\substack{j=1\\j\neq i}}^Na_j\nabla^\perp\,G_1\big(x_i(t)-x_j(t)\big)+\sum_{m=1}^Mc_m(t)\beta_m \big(x_i(t)\big).
\end{equation}

For a wider introduction to the Euler point-vortex system in bounded domains, we refer to~\cite[chap.\,4]{Marchioro_Pulvirenti_1984} or to~\cite{Newton_2001}. 
For more details on the Green's function in bounded domains, see \cite{Gustafsson_1979}.

%%%%%%%%%%%%%%%%%%%%%%%%%%%%%%%%%%%%%%%%%%%%%%%%%%%%%%%%%%%%%%%%%%%
\subsubsection{Second main result of the article}

The second main result of this article concerns the Hölder regularity for point-vortices in a smooth bounded domain:

\begin{theorem}[Hölder regularity for Euler point-vortices in bounded domains]\label{thrm:bounded}
Consider the Euler point-vortex dynamic \eqref{eq:evolution euler non simply connected} in a simply or multi-connected smooth domain $\Omega$. 
Assume that the intensities $a_i\neq 0$ satisfy  the \emph{non neutral clusters hypothesis} \eqref{eq:no null partial sum}. 
Consider an initial datum $X\in\RR^{2N}$ such that the dynamics is well-defined on $[0,T)$ with $T > 0$. Define the set of vortices that collapse with the boundary:
\begin{equation}\label{def:I}
    I:=\big\{i=1\dots N\,:\,\liminf\limits_{t\to T^-}\;\dist\big(x_i(t),\partial\Omega\big)=0\big\}.
\end{equation}

Then, there exists a constant $C$ such that the following properties hold true:

$(i)$ If $i\notin I$ then
\begin{equation}
    \forall\;t_1<t_2\in[0,T),\qquad\qquad|t_2-t_1|\leq1\quad\Longrightarrow\quad\big|x_i(t_2)-x_i(t_1)\big|\:\leq\:C\sqrt{t_2-t_1\,}.
\end{equation}
In particular, $x_i(t)$ converges as $t\to T^-$ towards an interior point of $\Omega.$\vspace{0.2cm}

$(ii)$ If $i\in I$ then
\begin{equation}
    \forall\;t_1<t_2\in[0,T),\qquad\qquad|t_2-t_1|\leq1\quad\Longrightarrow\quad\Big|\dist\big(x_i(t_2),\partial\Omega\big)-\dist\big(x_i(t_1),\partial\Omega\big)\Big|\:\leq\:C\sqrt{t_2-t_1\,}.
\end{equation}
In particular, the distance between $x_i(t)$ and $\partial\Omega$ converges to $0$ as $t\to T$.
\end{theorem}

Note that the Hölder regularity obtained for Euler in the plane (Theorem~\ref{thrm:plane} with $\alpha=1$) is the same as the one obtained for Euler in bounded domains (Theorem~\ref{thrm:bounded}). 
Here the Hölder constant depends on the intensities $a_i$, the circulations $\xi_m$, $\Omega$, $N$ but also on the final time $T$ and the initial data $X$, in a way that we will specify during the proof.

%%%%%%%%%%%%%%%%%%%%%%%%%%%%%%%%%%%%%%%%%%%%%%%%%%%%%%%%%%%%%%%%%%%
%%%%%%%%%%%%%%%%%%%%%%%%%%%%%%%%%%%%%%%%%%%%%%%%%%%%%%%%%%%%%%%%%%%
\subsection{General properties of the point-vortex systems.}
Before proving the $3$ theorems that we stated here-before,  we gather properties for the point-vortex systems that we use in the proofs or that give elements of context about our results.

%%%%%%%%%%%%%%%%%%%%%%%%%%%%%%%%%%%%%%%%%%%%%%%%%%%%%%%%%%%%%%%%%%%
\subsubsection{Hamiltonian formulation of the problem}
The first and main property of the point-vortex systems is their Hamiltonian structure. Indeed, if we define the Hamiltonian of the system for $X = (x_1,\ldots,x_N)$  by
\begin{equation}\label{def:Hamiltonien}
    H(X):=\frac{1}{2}\sum_{i\neq j}a_i\,a_j\,G_\alpha\big(x_i-x_j\big),
\end{equation}
then it is a direct computation from~\eqref{eq:evo alpha} to check that
\begin{equation}
    a_i\der{}{t}x_i(t)=\nabla_{x_i}^\perp H(X).
\end{equation}
In particular, the Hamiltonian $H$ is preserved by the flow associated to the differential equation~\eqref{eq:evo alpha}. 
With this Hamiltonian structure of the equation, it is possible to make use of the Noether theorem to determine the other quantities that are left invariant.
The invariance of the Hamiltonian with respect to the translations of the plane implies the preservation of the vorticity vector that is defined by
\begin{equation}\label{def:M}
    M(X):=\sum_{i=1}^Na_i\,x_i.
\end{equation}
Similarly, the invariance of the Hamiltonian with respect to the rotations of the plane implies the preservation of the vorticity momentum defined by
\begin{equation}\label{def:I(X)}
    I(X):=\sum_{i=1}^Na_i\,|x_i|^2
\end{equation}
These three conservation laws can also be obtained by a direct computation using the point-vortex equation~\eqref{eq:evo alpha} (see for instance~\cite{Godard-Cadillac_2021}).

%%%%%%%%%%%%%%%%%%%%%%%%%%%%%%%%%%%%%%%%%%%%%%%%%%%%%%%%%%%%%%%%%%%
\subsubsection{Non neutral clusters hypothesis}

In the reference work by Marchioro and Pulvirenti~\cite[chap.\,4]{Marchioro_Pulvirenti_1993}, the authors study the generic situation for which the intensities of any cluster is not equal to $0$, which we called the ``\emph{non-neutral clusters hypothesis}'' at~\eqref{eq:no null partial sum} (the term ``\emph{non-neutral}'' is chosen in analogy with Coulomb interaction~\cite[chap.\,4]{Marchioro_Pulvirenti_1993}).
The main reason behind such an hypothesis relies on the properties of the centers of vorticity. We denote by $\cP(N)$ the collection of all the subsets of $\{1,\dots, N\}$ and we define  \begin{equation}\cP_0(N):=\cP(N)\setminus\{\emptyset\}.\end{equation}
For $P\in\cP_0(N)$, called in this article a \emph{cluster of vortices}, we define the center of vorticity associated to the cluster $P$ as being the barycenter:
\begin{equation}\label{def:B_P}
    B_P := \bigg(\sum_{i\in P}a_i\bigg)^{-1}\sum_{i\in P} a_i\, x_i.
\end{equation}
If $P=\{1,\dots, N\}$, the barycenter is the center of vorticity of the whole system and it is simply denoted by $B$. As a consequence of the preservation of the center of vorticity~\eqref{def:M}, we have that $B$ is preserved by the flow. For the center of vorticity of the clusters, similar computations gives the following quasi-conservation property:
\begin{equation}\label{eq:quasi-preservation}
    \forall P \in \cP_0(N),\qquad \left| \der{}{t} B_P(t)\right| \leq \sum_{i\in P} \sum_{j\notin P} \frac{C}{|x_i(t)-x_j(t)|^\alpha},
\end{equation}
where $C$ is a constant depending on the intensities $a_i$. Indeed, using~\eqref{eq:evo alpha} and $x^\perp=-(-x)^\perp$:
\begin{equation}
     \der{}{t}\sum_{i\in P}a_i\,x_i(t)=\sum_{i\in P}\sum_{j\notin P} a_i\,a_j\,\frac{(x_i(t)-x_j(t))^\perp}{|x_i(t)-x_j(t)|^{\alpha+1}},
\end{equation}
In~\cite[chap.\,4]{Marchioro_Pulvirenti_1993} for Euler and in~\cite[Proposition 2.1]{Godard-Cadillac_2021} for the general case, the quasi-preservation of the centers of vorticity of the clusters is used as the main tool to prove the following uniform bound:

\begin{theorem}[Uniform bound on the trajectories~\cite{Godard-Cadillac_2021}]\label{thrm:borne uniforme}

Consider the point-vortex dynamic~\eqref{eq:evo alpha} under the non-neutral clusters hypothesis~\eqref{eq:no null partial sum} and with a kernel profile $G_\alpha$ for $\alpha\geq0$ fixed.

Then, given any positive time $T>0$, there exists a constant $C$ such that for any initial datum $X\in\RR^{2N}$ that is not leading to a collapse on $[0,T)$,
\begin{equation}\label{eq:borne uniforme}
\sup\limits_{t\in[0,T)}\big|X-S_\alpha^tX\big|\leq C,
\end{equation}
where $S^t_\alpha$ is the flow associated to~\eqref{eq:evo alpha} with kernel profile $G_\alpha$.
Moreover, the constant $C$ depends only on $N$, the intensities $a_i$, and on the final time $T$. This constant $C$ does not depend on the initial datum $X\in\RR^{2N}$ nor on $\alpha\geq0$.
\end{theorem}

A natural question concerning Theorem~\ref{thrm:borne uniforme} is to ask what this result becomes when the non-neutral cluster hypothesis~\eqref{eq:no null partial sum} ceases to be satisfied.
We focus here on the `\emph{non-neutral \emph{sub}-clusters hypothesis}'' defined as~\eqref{eq:no null sub partial sum}
In other words, all the strict sub-clusters must have the sum of their 
intensities different from $0$ but we allow the 
total sum $\sum_{i=1}^Na_i$ to be equal to $0$. This situation is 
achieved for instance by the vortex pair of intensities $+1$ 
and $-1$ that is translating at a constant speed. 
Having a total vorticity equal to $0$ corresponds to the physical situation where 
the fluid is initially at rest and, at $t=0$, it is subject to a vorticity-preserving perturbation.
In~\cite{Godard-Cadillac_2021} is established the following result:

\begin{theorem}[Uniform relative bound on the trajectories~\cite{Godard-Cadillac_2021}]\label{thrm:uniform relative bound}
For a given set of points noted $X=(x_1,\dots, x_N)\in\RR^{2N}$, we define the diameter of this set by
\begin{equation}
\diam(X)\;:=\;\max\limits_{i\neq j}|x_i-x_j|.
\end{equation}

Consider the point-vortex dynamic~\eqref{eq:evo alpha} under hypothesis~\eqref{eq:no null sub partial sum} with a kernel profile $G_\alpha$.
Let $T>0$ the final time. Then, for all initial datum $X\in\RR^{2N}$ that are not leading to collapse on $[0,T)$,
\begin{equation}\label{eq:borne uniforme relative}
\sup\limits_{t\in[0,T)}\diam\big(S_\alpha^tX\big)\leq \diam(X)+C,
\end{equation}
where $S^t_\alpha$ is the flow associated to~\eqref{eq:evo alpha}. Moreover, the constant $C$ depends only on $N$, the intensities $a_i$, and on the final time $T$. This constant $C$ does not depend on the initial datum $X\in\RR^{2N}$ nor on $\alpha>0$. 
\end{theorem}

The Hölder estimate that is proved in this article for the $\alpha$-point-vortex system (enunciated hereafter) must be considered in comparison with these two theorems.
More precisely, the theorem proved here can be interpreted in a certain way as an explicit computation of the constants appearing in~\eqref{eq:borne uniforme} and in~\eqref{eq:borne uniforme relative}, and more precisely concerning the dependency of these constants with respect to the final time $T$. It is indeed a consequence of the theorem proved hereafter that the constant $C$ in~\eqref{eq:borne uniforme} can be replaced by $\overline{C}\, T^\frac{1}{\alpha+1}$ when $T\leq 1$, where $\overline{C}$ is a constant independent on the final time $T$, and on the initial position $X\in\RR^{2N}$.

%%%%%%%%%%%%%%%%%%%%%%%%%%%%%%%%%%%%%%%%%%%%%%%%%%%%%%%%%%%%%%%%%%%
%%%%%%%%%%%%%%%%%%%%%%%%%%%%%%%%%%%%%%%%%%%%%%%%%%%%%%%%%%%%%%%%%%%
%%%%%%%%%%%%%%%%%%%%%%%%%%%%%%%%%%%%%%%%%%%%%%%%%%%%%%%%%%%%%%%%%%%
\section{Proofs of Hölder regularity in the plane: Theorems~\ref{thrm:plane} and~\ref{thrm:weak}.}
This section is devoted to the proof of the Hölder estimate in the plane stated in Theorems~\ref{thrm:plane} and~\ref{thrm:weak} for the $\alpha$-point-vortex dynamics, for any $\alpha \ge 1$.
The proof relies heavily on the quasi conservation of the center of vorticity of the different clusters of vortices stated in~\eqref{eq:quasi-preservation}. We first establish general properties for points $x_i$ that evolves in $\RR^p$ with $p \in \NN^\ast$ under general assumptions. In a second step, we show that the point-vortex system satisfies these assumptions with $p=2$ and this eventually leads to the conclusions of Theorems~\ref{thrm:plane} and~\ref{thrm:weak}.

%%%%%%%%%%%%%%%%%%%%%%%%%%%%%%%%%%%%%%%%%%%%%%%%%%%%%%%%%%%%%%%%%%%
%%%%%%%%%%%%%%%%%%%%%%%%%%%%%%%%%%%%%%%%%%%%%%%%%%%%%%%%%%%%%%%%%%%
\subsection{Clusters of vortices}
To start with, we define the degeneracy parameters
\begin{equation}
A_0:=\min_{\substack{P \in P(N) \\ P \neq \emptyset \\ P \neq \{1 ,\dots, N\}}} \left|\sum_{k\in P} a_k\right|,
\end{equation}
and
\begin{equation}\label{def:A}
A:=\min_{\substack{P \in P(N) \\ P \neq \emptyset}} \bigg|\sum_{k\in P} a_k\bigg|\;=\min\bigg\{A_0;\bigg|\sum_{i=1}^Na_i\bigg|\bigg\}.
\end{equation}
Remark that $A_0>0$ is equivalent to the \emph{non-neutral sub-clusters hypothesis}~\eqref{eq:no null sub partial sum} 
and $A>0$ is equivalent to the \emph{non-neutral clusters hypothesis}~\eqref{eq:no null partial sum}.
We also denote:
\begin{equation}\label{def:a}
a_0:=\bigg|\sum_{i=1}^Na_i\bigg|,\qquad\text{and}\qquad a:=\sum_{i=1}^N|a_i|.
\end{equation}
We are in position to state the first lemma:

\begin{lemma}\label{coro:B_control}
Let $(x_i)_{1\leq i \leq N}$ be a family of $N \in \NN^\ast $ points of $\RR^p$ associated to intensities $(a_i)_{1\leq i \leq N}$ satisfying the \emph{non neutral sub cluster hypothesis}~\eqref{eq:no null sub partial sum}. Then for every $P \in P(N)$ such that $P\neq \{1,\ldots,N\}$ and $P\neq \emptyset$, we have for all $i\in P$,
\begin{equation}
    |x_i - B_P| \;\leq\; \frac{a}{A_0}\;\max\limits_{j \in P} |x_i-x_j|,
\end{equation}
where $B_P$ was defined in~\eqref{def:B_P}.
\end{lemma}

\begin{proof}
The definition of $B_P$ gives
\begin{equation}
x_i-B_P=\frac{\sum_{j\in P}a_j(x_i-x_j)}{\sum_{j\in P}a_j}.
\end{equation}
Then, using the triangular inequality,
\begin{equation}
|x_i - B_P| \;\leq\;\frac{\sum_{j\in P}|a_j|\,\big|x_i-x_j\big|}{\big|\sum_{j\in P}a_j\big|}\;\leq\; \frac{\sum_{j\in P}|a_j|}{ \big|\sum_{j\in P} a_j\big|}\;\max\limits_{j\in P} |x_i-x_j|.
\end{equation}
Recalling the definitions of $a$ and $A_0$ given previously ends the proof.
\end{proof}

\begin{remark}\label{the_remark}
In the previous Lemma, if we assume the stronger hypothesis \eqref{eq:no null partial sum}, then the conclusion is also true for $P = \{1,\ldots,N\}$ by replacing $A_0$ by $A$.
\end{remark}

To study the system of vortices, we separate them into clusters so that we can study separately the sets of vortices that are close to each-other. 
The objective is to have both a control on the distances between vortices inside a same cluster and a control of the distance between clusters.
To build these clusters, we make use of the following lemma that is reminiscent of Lemma~A.2 in~\cite{Bethuel_Orlandi_Smets_2007_1}.

\begin{lemma}[Balls lemma]\label{lem:balls}
Let $(x_i)_{1\leq i \leq N}$ be $N$ points in $\RR^p$. Let $\eps > 0$ and $0<\kappa\leq \frac{1}{2}$. We denote by $\cB(x,r)$ the ball of $\RR^p$ centered in $x$ of radius $r$. Then there exist $\delta>0$
and a set of indices $Q \in \cP_0(N)$ such that
\begin{equation}\label{eq:interval for delta}
    \eps\leq\delta<\bigg(\frac{\kappa}{2}\bigg)^{-N}\!\eps,
\end{equation}
\begin{equation}\label{eq:inside clusters}
    \bigcup_{i=1}^N \cB(x_i,\eps) \subset \bigcup_{j\in Q} \cB(x_j,\delta)
\end{equation}
and
\begin{equation}\label{eq:outside clusters}
    \forall\;i\neq j \in Q, \quad |x_i-x_j|\geq 
    \kappa^{-1}\delta.
\end{equation}
\end{lemma}
\begin{proof}
First, define the set $Q_1:=\{1,\dots,N\}$ and the parameter $\delta_1:=\eps>0$. 
Obviously, with these definitions, Property~\eqref{eq:inside clusters} is satisfied. 
We then proceed by iteration and build a sequence of sets $Q_k$ (with $Q_{k}\subseteq Q_{k-1}$) for $k=1\dots K$.
We also define the parameters $\delta_k=\eps(\kappa/2)^{-k+1}$. 
Suppose that the set $Q_k$ has already been built for some $k$ and that this set satisfies~\eqref{eq:inside clusters} with parameter $\delta_k$, but does not satisfy~\eqref{eq:outside clusters}.
Then there exists $x_i \neq x_j$ in $Q_k$ such that $|x_i-x_j|<\kappa^{-1}\delta_k$.
We define $Q_{k+1}:=Q_k\setminus\{j\}$ and we recall that $\delta_{k+1}=(\kappa/2)^{-1}\delta_k$. We next prove: 
\begin{equation}
\bigcup_{\ell\in Q_k} \cB(x_\ell,\delta_k)\subset \bigcup_{l\in Q_{k+1}} \cB(x_\ell,\delta_{k+1})
\end{equation}
which implies that Property~\eqref{eq:inside clusters} is still satisfied at step $k+1$. To prove the relation above, it suffices to establish that $\cB(x_j,\delta_k)\subseteq \cB(x_i,\delta_{k+1})$. This indeed holds true since for all $x\in \cB(x_j,\delta_k)$ we have that $|x-x_i|\leq |x-x_j|+|x_i-x_j|\leq \delta_k+\kappa^{-1}\delta_k<(\kappa/2)^{-1}\delta_k=\delta_{k+1}$.

The sequence $(Q_k)$ is such that $Q_k$ has its cardinal equal to $N+1-k$. There exists necessarily a step $K\leq N$ such that~\eqref{eq:outside clusters} holds. Indeed, if this is not the case for all $k=1\dots N-1$, then at the $N^{th}$ step the set $Q_N$ is reduced to one element and thus~\eqref{eq:outside clusters} is satisfied since the condition is void. 
In particular, this construction requires  at most $N$ iterations and therefore Property~\eqref{eq:interval for delta} is satisfied with $\delta_k$ for all $k$.
This eventually concludes the proof of Lemma~\ref{lem:balls} by setting $\delta:=\delta_K$ and $Q:=Q_K$.
\end{proof}
We deduce the following Corollary that we will use later in this article.
\begin{corollary}\label{coro:balls}
Let $(x_i)$ be a family of $N$ points in $\RR^p$. Then for all $\kappa \in (0,1)$ and for all $d > 0$, there exist $\delta \in [\left(\frac{\kappa}{8}\right)^N\!\!d,d)$ and a partition $\mP$ of the set $\{1,\dots, N\}$ such that
\begin{equation}\label{eq:cluster inside}
    \forall\, P \in \mP, \quad \forall\, i,j \in P, \qquad |x_i - x_j| \leq \delta
\end{equation}
and
\begin{equation}\label{eq:cluster outside}
    \forall\, P \neq P' \in \mP, \quad \forall\, i \in P, \quad \forall\, j \in P', \qquad |x_i - x_j| \ge \kappa^{-1} \delta.
\end{equation}
\end{corollary}

\begin{proof}
Let $0<\kappa < 1$. We set $\eps =\frac{1}{2}\left(\frac{\kappa}{8}\right)^N\!\!d$ and $\kappa' = (2\kappa^{-1}+2)^{-1}$. Then $\kappa' \in (0,\frac{1}{4})$ and we can apply Lemma~\ref{lem:balls} with $\eps$ and $\kappa'$. We then have a $\delta'$ and a set $Q\subseteq\{1,\dots, N\}$ such that
\begin{equation}\label{eq:interval for delta V2}
    \eps\leq\delta'<\bigg(\frac{\kappa'}{2}\bigg)^{-N}\!\eps,
\end{equation}
\begin{equation}\label{eq:inside clusters V2}
    \bigcup_{i=1}^N \cB(x_i,\eps) \subset \bigcup_{j\in Q} \cB(x_j,\delta')
\end{equation}
and
\begin{equation}\label{eq:outside clusters V2}
    \forall\;i\neq j \in Q, \quad |x_i-x_j|\geq 
    (\kappa')^{-1}\delta'.
\end{equation}

Let $\delta := 2\delta'$. Property~\eqref{eq:interval for delta  V2} gives
\begin{equation}
\left(\frac{\kappa}{8}\right)^N\!\!d\leq \delta<\bigg(\frac{\kappa'}{2}\bigg)^{-N}\!\left(\frac{\kappa}{8}\right)^N\!\!d.
\end{equation} 
Since $\frac{\kappa}{\kappa'} = 2 + 2\kappa < 4$, we have $\delta \in [\left(\frac{\kappa}{8}\right)^N\!\!d,d)$. 

For all $i \in Q$, we set
\begin{equation}
    P_i = \big\{ j \in \{1,\dots, N\}, \quad |x_i-x_j| \leq \delta'\big\}.
\end{equation}
Let $\mP = \{ P_i,\, i \in Q \}$. 
The definition of $P_i$ gives~\eqref{eq:cluster inside}. 
Relation~\eqref{eq:outside clusters V2} and the definition of $P_i$ imply that
\begin{equation}
    \forall\, i_1\neq i_2\in Q, \ \forall\, j \in P_{i_1}, \ \forall\, k \in P_{i_2}, \quad |x_j - x_k|\ge |x_{i_1}-x_{i_2}|-|x_j-x_{i_1}|-|x_k-x_{i_2}| \ge (\kappa')^{-1} \delta' - 2 \delta' = \kappa^{-1} \delta,
\end{equation}
which is~\eqref{eq:cluster outside}. 
Finally, relation~\eqref{eq:inside clusters V2} implies that every index $i \in \{1 ,\dots, N\}$ belongs to at least one element of $\mP$ and the relation above gives that the elements of $\mP$ are pairwise disjoints.
Therefore $\mP$ is indeed a partition of $\{1,\dots, N\}$.
\end{proof}

%%%%%%%%%%%%%%%%%%%%%%%%%%%%%%%%%%%%%%%%%%%%%%%%%%%%%%%%%%%%%%%%%%%
%%%%%%%%%%%%%%%%%%%%%%%%%%%%%%%%%%%%%%%%%%%%%%%%%%%%%%%%%%%%%%%%%%%
\subsection{Sufficient condition to prevent a collapse}
The Hölder regularity result can be seen as a necessary condition to have a collapse of point-vortices. The strategy is then to investigate the sufficient conditions that prevent a collapse. Within this approach, we proved the following general proposition:

\begin{proposition}[Sufficient condition to prevent a collapse]\label{prop:prevent collapse}
For $1 \leq i \leq N$, let $x_i$ be a family of $N$ different points of $\RR^p$ evolving on a time interval $[0,T)$, $T>0$.
Let $(a_i)_{1\le i \leq N} \in \RR^N$ satisfy~\eqref{eq:no null sub partial sum}. Recall the definition of the center of vorticity of clusters  with intensities $a_i$ given at~\eqref{def:B_P}.
We make the hypothesis that the evolution of these points is weakly differentiable and such that there exists $C_0, C_1 \ge 0$ and $\alpha\geq 0$ such that for all $t \in [0,T)$,
\begin{equation}\label{hyp:slow center of vorticity}
    \forall P \in \cP_0(N)\setminus\{1\ldots N\},\qquad \left| \der{}{t} B_P(t)\right| \leq \sum_{i\in P} \sum_{j\notin P} \frac{C_0}{|x_i(t)-x_j(t)|^\alpha} + C_1.
\end{equation}

Then there exists a constant $C_2>0$ such that for all $\eta \in (0,1]$, for all $t \in [0,T)$ satisfying
\begin{equation}
    T-t \le C_2\, \eta^{\alpha+1},
\end{equation}
and for all indices $i,j \in \{1,\dots, N\}$,  the following implication is true:
\begin{equation}
    |x_i(t) - x_j(t)| \geq \eta \quad\Longrightarrow\quad \forall \tau \in [t,T), \quad |x_i(\tau)-x_j(\tau)| \ge \frac{\eta}{2}.
\end{equation}
The constant $C_2$ depends only on $\alpha$, $a$, $A_0$, $C_0$, $C_1$ and $N$.
\end{proposition}

The idea of the proof is the following. We group the points into clusters. We observe that past a certain time, if a point start moving in a significant manner, it means that its cluster has to spread and divide itself into smaller clusters. Indeed this would otherwise contradict the preservation of the center of vorticity. Since there is a finite number of points, spreading of clusters can happen only a finite number of times. 
\begin{proof}
We can assume without loss of generality that $\max\{C_0,C_1\}>0$. We fix once and for all some $\eta\in(0,1]$. We also fix two indices $i\neq j$ and a time $t_1$. We assume that  $|x_i(t_1) - x_j(t_1)| \geq \eta $ and 
\begin{equation}\label{hyp:t1}
    T-t_1 \le C_2\, \eta^{\alpha+1}
\end{equation} 
for some constant $C_2$ to be chosen later. During the proof, we will impose several conditions on the constant $C_2$ and at the end of the proof we will observe that all these conditions can be satisfied for a constant $C_2$ which is independent of $\eta$.\vspace{0.2cm}

\emph{Step 1: Iterative construction of a sequence of clusters.}
We recall that $A_0$ and $a$ are defined by~\eqref{def:A} and~\eqref{def:a} respectively. By hypothesis~\eqref{eq:no null sub partial sum}, we have that $A_0>0$. Let $0<\kappa < \frac{A_0}{16a}$. Remark that $\kappa\leq1/16$ since $A_0\leq a$.
We are now building partitions of $\{1,\dots, N\}$ with an iterative process. 

We first invoke the corollary of the balls lemma (Corollary~\ref{coro:balls}) to the points $x_k(t_1)$ to build the first partition $\mP^1$ by choosing $d := \kappa \eta$.
This gives the first partition $\mP^1$ and a real number $\delta_1$ satisfying 
\begin{equation}\label{def:delta_1}
    \left(\frac{\kappa}{8}\right)^N\kappa\eta\leq \delta_1 < \kappa\eta
\end{equation} 
such that
\begin{equation}
\forall\; P \in \mP^1,\quad \forall\; k,\ell \in P, \qquad |x_k(t_1)-x_\ell(t_1)| \leq \delta_1
\end{equation}
and
\begin{equation}
\forall P \neq P' \in \mP^1,\quad \forall k \in P, \quad \forall \ell \in P', \quad |x_k(t_1)-x_\ell(t_1)| \ge \kappa^{-1}\delta_1.
\end{equation}
The elements of such a partition are called \emph{clusters}. 
Since $\delta_1 < \eta$, the indices $i$ and $j$ do not belong to the same cluster. 
In particular $\{1,\dots, N\} \notin \mP^1$.  
We now define
\begin{equation}\label{def:r}
    r := \min \left\{ \frac{1}{8}\,;\,\frac{A_0}{8a\kappa}-2\right\}>0,
\end{equation}
and
\begin{equation}
    s := r\left(\frac{\kappa}{8}\right)^N.
\end{equation}
Note that $0<s<r<1$.

We now build iteratively a finite number of partitions denoted by $\mP^q$, $q \ge 1$, a finite sequence of positive numbers $\delta_q$ and an increasing finite sequence of times $(t_q)$ that satisfy the following properties:\vspace{0.1cm}

        \begin{equation}\label{(ii_v1)}\tag{$i$}
        \forall\; P \in \mP^q, \quad \forall\;k,\ell \in P, \qquad |x_k(t_q)-x_\ell(t_q)| \leq \delta_q.
        \end{equation}     \vspace{0.1cm}
and\vspace{0.1cm}
        \begin{equation}\label{(iii_v1)}\tag{$ii$}
        \forall P \neq P' \in \mP^q, \quad \forall k \in P, \quad \forall \ell \in P', \quad |x_k(t_q)-x_\ell(t_q)| \geq \kappa^{-1} \delta_q.
        \end{equation}\vspace{0.1cm}
The construction proceeds as follows. If the following relation is satisfied:
        \begin{equation}\label{(vi_v1)}\tag{$\ast$}
           \forall\;k=1\dots N,\quad  \forall \tau \in [t_q,T), \qquad |x_k(\tau)-x_k(t_q)| \leq \kappa^{-1}\delta_q/8.
        \end{equation}\vspace{0.1cm}
then the construction stops. Else, we will construct $t_{q+1} \in (t_q,T)$ such that
        \begin{equation}\label{(v_v1)}\tag{$iii$}
           \forall\;k=1\dots N,\quad \forall \tau \in [t_q,t_{q+1}], \qquad |x_k(\tau)-x_k(t_q)| \leq \kappa^{-1}\delta_q/8.
        \end{equation}\vspace{0.1cm}
and $\delta_{q+1}$ such that
        \begin{equation}\label{(iv_v1)}\tag{$iv$}
             s\delta_q \leq\delta_{q+1} <  r \delta_q.
        \end{equation} \vspace{0.1cm}
and the next partition $\mP^{q+1}$ satisfying \eqref{(ii_v1)} and \eqref{(iii_v1)} at step $q+1$ and
\begin{equation}\label{(i_v1)}\tag{$v$}
    \mP^{q+1} \text{ is a strict sub-partition of } \mP^{q}.
\end{equation}\vspace{0.1cm}

Notice that $\mP^1$ satisfies properties~\eqref{(ii_v1)} and~\eqref{(iii_v1)}. 
Let $q\in\NN^\ast$ be fixed and assume now that the partitions $\mP^{q'}$ are built for all $q'=1\dots q$. Assuming that~\eqref{(vi_v1)} is not satisfied, we proceed to construct $t_{q+1}$, $\delta_{q+1}$ and $\mP^{q+1}$.

By continuity of the trajectories, we denote by $t_{q+1} \in [t_q,T)$  the largest time such that
\begin{equation}\label{majtq+1 jusqu au bout}
\forall\;\tau\in[t_q,t_{q+1}],\quad\forall\;\ell=1\dots N,\qquad |x_\ell(\tau) - x_\ell(t_{q})| \leq \kappa^{-1}\delta_q/8,
\end{equation}
which is correctly defined since~\eqref{(vi_v1)} is not satisfied. Then~\eqref{(v_v1)} holds true. 

One uses now Corollary~\ref{coro:balls} with parameter $d=  r\delta_q$ and with the family of points given by $x_i(t_{q+1})$ for $i=1\dots N$. 
This gives $\delta_{q+1}$ such that $\left(\frac{\kappa}{8}\right)^N d\leq\delta_{q+1}< d$ and a partition $\mP^{q+1}$ that satisfies\vspace{0.2cm}

\begin{equation}
\forall\; P \in \mP^{q+1},\quad \forall\; m,n \in P, \qquad |x_m(t_{q+1})-x_n(t_{q+1})| \leq  \delta_{q+1},
\end{equation}
and
\begin{equation}\forall\; P \neq P' \in \mP^{q+1},\quad \forall m \in P,\quad \forall\; n \in P', \qquad |x_m(t_{q+1})-x_n(t_{q+1})| \geq \kappa^{-1} \delta_{q+1}\end{equation}
which shows conditions~\eqref{(ii_v1)} and~\eqref{(iii_v1)} at step $q+1$.
Recalling that $d = r\delta_q$ and $s = \left(\frac{\kappa}{8}\right)^N r$ gives that $s\delta_q \le \delta_{q+1}< r \delta_q$ so that condition~\eqref{(iv_v1)} is proved.

It remains to show condition~\eqref{(i_v1)}. Using~\eqref{(iii_v1)} and~\eqref{majtq+1 jusqu au bout}, one has for any $P \in \mP^q$, for all $\tau \in [t_q,t_{q+1}]$, for any $m\in P$ and $n\notin P$ that
\begin{equation}\label{eq:majuv}
\begin{split}
    |x_m(\tau)-x_n(\tau) | & = |x_m(\tau)-x_m(t_q) + x_m(t_q) - x_n(t_q) + x_n(t_q) - x_n(\tau)| 
    \\ & \ge | x_m(t_q) - x_n(t_q)| - |x_m(\tau)-x_m(t_q)| -  |x_n(t_q) - x_n(\tau)| 
    \\ & \ge \kappa^{-1}\delta_q ( 1 - 2/8) \\ & \ge \kappa^{-1}\delta_q / 2
    \\ & > \delta_{q+1}
\end{split}
\end{equation}
since $\kappa^{-1}/2 > 1 > r$ and $\delta_{q+1} < r\delta_q$ by~\eqref{(iv_v1)}. 
This estimate applied at time $\tau = t_{q+1}$ implies that such two indexes $m$ and $n$ do not belong to the same cluster in $\mP^{q+1}$, otherwise it contradicts~\eqref{(ii_v1)}. Therefore two indexes belonging to two different clusters in $\mP^q$ still belong to two different clusters in $\mP^{q+1}$. This proves that $\mP^{q+1}$ is a sub-partition of $\mP^q$. 

We now prove that it is a strict sub-partition. Since $t_{q+1}$ is the largest time such that \eqref{majtq+1 jusqu au bout} holds, there exists at least one index $k \in \{1,\ldots,N\}$ such that
\begin{equation}\label{majtq+1}
  |x_k(t_{q+1}) - x_k(t_{q})| = \kappa^{-1}\delta_q/8.
\end{equation}
Let $P \in \mP^q$ such that $k \in P$. 
Note that since $\mP^1$ is not the trivial partition $\{\{1,\dots, N\}\}$, neither is $\mP^q$ since $\mP^q$ is a sub-partition of $\mP^{1}$ as shown above.
Thus, $P \neq \{1 ,\dots, N\}$. Therefore thanks to Hypothesis~\eqref{eq:no null sub partial sum} we can apply Lemma~\ref{coro:B_control} to obtain:
\begin{equation}\label{majBtq+1}
   |x_k(t_{q+1}) - B_P(t_{q+1})| \leq \frac{a}{A_0} \max_{\ell\in P } |x_k(t_{q+1})-x_\ell(t_{q+1})|,
\end{equation}
and
\begin{equation}\label{majBtq}
    |x_k(t_{q}) - B_P(t_{q})| \leq \frac{a}{A_0} \max_{\ell\in P } |x_k(t_{q})-x_\ell(t_{q})|\leq \frac{a}{A_0} \delta_q,
\end{equation}
where for the last inequality one uses the recursive hypothesis~\eqref{(ii_v1)}.
Recall here hypothesis~\eqref{hyp:slow center of vorticity}:
\begin{equation}\label{eq:maj dbp}
    \left|\der{}{t} B_P(t)\right| \leq \sum_{m\in P} \sum_{n\notin P} \frac{C_0}{|x_m(t)-x_n(t)|^\alpha} +C_1.
\end{equation}
Recall that by relation \eqref{eq:majuv}, if $P\neq P' \in \mP^q$, then for $m \in P$, $n \in P'$ and $\tau \in [t_q,t_{q+1}]$ we have
\begin{equation}
|x_m(\tau) - x_n(\tau)| \geq \kappa^{-1}\delta_q/2.
\end{equation}
Plugging this into \eqref{eq:maj dbp} gives for all $\tau \in [t_q,t_{q+1}]$,
\begin{equation}\label{de Cadix a des yeux de velours}
    \left|\der{}{t} B_P(\tau)\right| \leq \frac{2^\alpha\, N^2\,C_0}{(\kappa^{-1}\delta_q)^\alpha}+C_1.
\end{equation}
Thus,
\begin{equation}\label{bio}
    |B_P(t_{q+1}) - B_P(t_q)|\leq \left(\frac{2^\alpha\, N^2\,C_0}{(\kappa^{-1}\delta_q)^\alpha}+ C_1\right)(t_{q+1}-t_q).
\end{equation}

We assume that the constant $C_2$ is small enough such that 
\begin{equation}\label{j ai soif}
C_2\eta^{\alpha+1}\leq \frac{a}{A_0} \delta_q \left(\frac{2^\alpha\, N^2\,C_0}{(\kappa^{-1}\delta_q)^\alpha}+ C_1\right)^{-1}.
\end{equation}
Combining this with relation \eqref{hyp:t1}, we have that
\begin{equation}
        |T - t_1| \leq C_2\eta^{\alpha+1}\leq \frac{a}{A_0} \delta_q \left(\frac{2^\alpha\, N^2\,C_0}{(\kappa^{-1}\delta_q)^\alpha}+ C_1\right)^{-1}.
\end{equation}
Thus, since $t_1\leq t_q\leq t_{q+1}< T$,
\begin{equation}
    |t_q - t_{q+1}| \leq \frac{a}{A_0} \delta_q \left(\frac{2^\alpha\, N^2\,C_0}{(\kappa^{-1}\delta_q)^\alpha}+ C_1\right)^{-1}.
\end{equation}
Therefore, plugging this in~\eqref{bio},
\begin{equation}\label{majEvolB}
    |B_P(t_{q+1}) - B_P(t_q)|\leq \frac{a}{A_0}\delta_q.
\end{equation}

Gathering now \eqref{majtq+1}, \eqref{majBtq} and  \eqref{majEvolB} gives 
\begin{align*}
    |x_k(t_{q+1}) - B_P(t_{q+1})| & = |x_k(t_{q+1}) - x_k(t_q) + x_k(t_q) - B_P(t_q) + B_P(t_q) -B_P(t_{q+1})| \\
    & \ge |x_k(t_{q+1}) - x_k(t_q)| - |x_k(t_q) - B_P(t_q)| - |B_P(t_q) -B_P(t_{q+1})|\\
    & \ge (\kappa^{-1}/8 - 2a/A_0)\delta_q.
\end{align*}
The estimate \eqref{majBtq+1} yields the fact that there exists an index $\ell \in P$ such that
\begin{equation}
    |x_k(t_{q+1})-x_\ell(t_{q+1})| \geq \frac{A_0}{a} |x_k(t_{q+1}) - B_P(t_{q+1})| 
\end{equation}
and combining this with the previous estimates gives
\begin{equation}\label{eq:points s'eloignent}
    |x_k(t_{q+1})-x_\ell(t_{q+1})| \geq \frac{A_0}{a}(\kappa^{-1}/8 - 2a/A_0)\delta_q \geq r\,\delta_q,
\end{equation}
since $r$ is defined by~\eqref{def:r}.
Since $r\delta_q>0$, this implies in particular that $k\neq\ell$ and therefore $P$ has at least two elements.

%%%%%%%%%%%%%%%

Using $\delta_{q+1} < r\delta_q$ and relation \eqref{eq:points s'eloignent}, we have that $\delta_{q+1}< |x_k(t_{q+1})-x_\ell(t_{q+1})|$. 
Thus the two indices $k$ and $\ell$ do not belong to the same cluster anymore in $\mP^{q+1}$, since~\eqref{(ii_v1)} is satisfied at step $q+1$.
Consequently, $\mP^{q+1}$ is a strict sub-partition of $\mP^q$. All the recursive properties are proved. This concludes our construction.

We observe now that at some step $Q$ condition~\eqref{(vi_v1)} must be verified. Indeed, since at every step the new partition is a strict sub-partition of the previous one, the number of clusters in the partition is a strictly increasing sequence, starting from at least $2$ elements and cannot exceed $N$. We insist on the fact that we proved that if~\eqref{(vi_v1)} is not satisfied then we \emph{can} construct the next strict sub-partition. Since it is not possible to iterate more than $N-2$ times, relation~\eqref{(vi_v1)} is achieved after at most $N-2$ steps. We thus have $Q \leq N-1$, with condition \eqref{(vi_v1)} satisfied at step $Q$.\vspace{0.2cm}

\emph{Step 2: conclusion of the proof of the Proposition.}
We now prove that for every $\tau \in [t_1,T)$, $|x_i(\tau)-x_j(\tau)| \ge \eta/2$.

Let us start by noticing that hypothesis~\eqref{(iv_v1)} of our construction gives by induction that for all $q \in \{1,\dots, Q\}$ we have
\begin{equation}
    s^{q-1} \delta_1 \le \delta_q \le r^{q-1} \delta_1.
\end{equation}
Recalling that $\delta_1$ satisfies~\eqref{def:delta_1}, and the fact that $s < \left(\frac{\kappa}{8}\right)^N$, we have that
\begin{equation}\label{eq:encadrement_delta}
    s^q \kappa \eta \le \delta_q \le r^{q-1} \kappa \eta.
\end{equation}
Let $\tau \in [t_1,T)$, and we set $t_{Q+1} = T$ so that there exists a unique $q \in \{1,Q\}$ verifying $\tau \in [t_q,t_{q+1})$. We have that for all $k \in \{1,\ldots,N\}$,
\begin{equation}
    |x_k(\tau) - x_k(t_1)| \leq  \sum_{q'=1}^{q-1} |x_k(t_{q'+1})-x_k(t_{q'})|+ |x_k(\tau) - x_k(t_q)|.
\end{equation}
Then, by the construction hypothesis~\eqref{(v_v1)} and \eqref{(vi_v1)} we have that
\begin{equation}
    |x_k(\tau) - x_k(t_1)| \leq \sum_{q'=1}^{q}\kappa^{-1}\delta_{q'}/8.
\end{equation}
Recall from \eqref{def:r} that $r<1/8$. By relation \eqref{eq:encadrement_delta} this yields for all $k$ and for all $\tau \ge t_1$,
\begin{equation}
    |x_k(\tau) - x_k(t_1)| \leq \frac{1}{8}\sum_{q'=1}^q r^{q'-1} \eta\leq \eta \sum_{q'=1}^{q}\bigg(\frac{1}{8}\bigg)^{q'}\leq \eta/4.
\end{equation}
Therefore, 
\begin{align*}
    |x_i(\tau)-x_j(\tau)| & = |x_i(\tau) - x_i(t_1) + x_i(t_1)-x_j(t_1) + x_j(t_1) - x_j(\tau)| \\
    & \ge |x_i(t_1)-x_j(t_1)| - |x_i(\tau) - x_i(t_1)| - |x_j(t_1) - x_j(\tau)| \\
    & \ge \eta - 2\eta/4 \\
    & \ge \eta/2.
\end{align*}

To conclude the proof of Proposition~\ref{prop:prevent collapse}, one still has to check that $C_2$ can be chosen independently of $\eta$. The constant $C_2$ must satisfy Condition~\eqref{j ai soif} for all indices $q \in \{ 1,\dots, Q\}$. This leads to the following condition:
\begin{equation}\label{eq:condition t reformulee}
    C_2 \eta^{\alpha+1} \leq \min_{q\in \{1, \dots, Q\} } \frac{a}{A_0} \delta_q \left(\frac{2^\alpha\, N^2\,C_0}{(\kappa^{-1}\delta_q)^\alpha}+ C_1\right)^{-1}.
\end{equation}

Since $\delta_q$ is of size $\eta$ and it is the only quantity depending on $\eta$, we observe that both sides of this relation are of size $\eta^{\alpha+1}$. So the existence of $C_2$ independent of $\eta$ is clear. More precisely, since $\kappa^{-1}\delta_q < \eta \le 1$ by \eqref{eq:encadrement_delta} and by the fact that $r\le 1$, we have that
\begin{equation}
    C_1 \le \frac{C_1}{(\kappa^{-1}\delta_q)^\alpha}.
\end{equation}
Therefore, it is enough to take $C_2$ such that
\begin{equation}
    C_2 \eta^{\alpha+1} \leq \left(\min_{q\in \{1, \dots, Q\} }  \delta_q^{\alpha+1}\right) \frac{a}{A_0} \frac{\kappa^{-\alpha}}{2^\alpha\, N^2\,C_0+C_1}.
\end{equation}
Using relation \eqref{eq:encadrement_delta} and the fact that $Q \le N-1$, we have that 
\begin{equation}
    \delta_q^{\alpha+1} \ge (s^{N-1} \kappa \eta)^{\alpha+1}.
\end{equation}
This proves that if we chose
\begin{equation}
    C_2 = \frac{a}{A_0} s^{(N-1)(\alpha+1)}\frac{\kappa}{2^\alpha\, N^2\,C_0+C_1}
\end{equation}
then relations \eqref{j ai soif} are satisfied for all $q = 1\dots Q$. 
Replacing $s$ by its value and taking for example $\kappa = \frac{A_0}{17a}$  eventually gives the announced constant $C_2$ in Proposition~\ref{prop:prevent collapse}. 
Note that $C_2$ is positive and is finite since $\max\{C_0,C_1\}>0$. 
This constant is given explicitly and depends only on $a$, $A_0$, $\alpha$, $C_0$, $C_1$ and $N$.
\end{proof}

%%%%%%%%%%%%%%%%%%%%%%%%%%%%%%%%%%%%%%%%%%%%%%%%%%%%%%%%%%%%%%%%%%%
%%%%%%%%%%%%%%%%%%%%%%%%%%%%%%%%%%%%%%%%%%%%%%%%%%%%%%%%%%%%%%%%%%%
\subsection{Conclusion of the proofs}

%%%%%%%%%%%%%%%%%%%%%%%%%%%%%%%%%%%%%%%%%%%%%%%%%%%%%%%%%%%%%%%%%%%
\subsubsection{Convergence lemmas}
With this proposition at hand, we are in position to state the lemmas which will allow to conclude the proofs of Theorems~\ref{thrm:plane} and~\ref{thrm:weak}.
We first prove the following convergence result on the relative dynamics $t\mapsto x_i(t)-x_j(t):$

\begin{lemma}[Clusters of points going to collision]\label{lem:cluster}
Let $x_i(t)\in\RR^p$ be a set of points evolving in time on a time interval $[0,T)$. We assume that for all $t \in [0,T)$ and $i \neq j$, $x_i(t) \neq x_j(t)$. Consider intensities $a_i$ associated to these points such that the non-neutral sub-clusters hypothesis \eqref{eq:no null sub partial sum} holds true. Assume that we have~\eqref{hyp:slow center of vorticity}  for some $\alpha\geq 0$. Then there exists a partition $\mP$ of the set $\{1 ,\dots, N \}$ and a distance $\delta > 0$ such that
\begin{equation}\label{eq:pointsCollide}
    \forall P \in \mP,\quad \forall\, i,j \in P, \qquad \lim_{t \rightarrow T^-} |x_i(t)-x_j(t)| = 0
\end{equation}
and
\begin{equation}\label{eq:pointsStayFar}
    \forall P \neq P' \in \mP, \quad \forall i \in P, \quad \forall j \in P', \quad \forall t \in [0,T), \qquad |x_i(t)-x_j(t)| \ge \delta.
\end{equation}
\end{lemma}
\begin{proof}
We observe that the relation $i\sim j$ if $\lim_{t \rightarrow T^-} |x_i(t)-x_j(t)|=0$ is an equivalence relation. We define the partition $\mP$ as the equivalence classes associated to this equivalence relation. Then relation \eqref{eq:pointsCollide} is trivially verified. We prove now that relation \eqref{eq:pointsStayFar} is also true. 

It suffices to show that if there are two indices $i\neq j$ such that \begin{equation}\liminf_{t\rightarrow T^-} |x_i(t)-x_j(t)| = 0,\end{equation} then
\begin{equation}
    \lim_{t\rightarrow T^-} |x_i(t)-x_j(t)| = 0.
\end{equation} 
Assume that $\liminf_{t\rightarrow T^-} |x_i-x_j| = 0$ and for the sake of contradiction that $|x_i-x_j|$ doesn't converge to $0$. We infer that there exists $\eta \in (0,1]$ and two sequences $t_n<t'_n<T$ of times going to $T$ as $n$ goes to infinity such that 
\begin{equation}\begin{cases}
    |x_i(t_n)-x_j(t_n)|\geq \eta \\
    |x_i(t'_n)-x_j(t'_n)| < \eta/2.
    \end{cases}
\end{equation}
According to  Proposition~\ref{prop:prevent collapse}, this cannot hold as soon as $T-t_n \leq C_2\eta^{\alpha+1}$. We have our contradiction.
\end{proof}
The elements of $\mP$ are called the~\emph{clusters of collisions} since $i,j$ belong to the same $P$ if and only if the associated point-vortices $x_i(t)$ and $x_j(t)$ are going to collide. With this convergence property, we can state the final lemma:

\begin{lemma}[Hölder estimate lemma]\label{lem:traj are holder}
Let $x_i(t)\in\RR^p$ be a set of points evolving on an interval of time $[0,T)$ with $T\leq 1$ such that for all $t \in [0,T)$ and $i \neq j$, we have $x_i(t) \neq x_j(t)$. Associate to these points their intensities $a_i$ such that the non-neutral sub-clusters hypothesis \eqref{eq:no null sub partial sum} holds true.
Let $\alpha \ge 0$ and assume that the system satisfies~\eqref{hyp:slow center of vorticity} with some constants $C_0$, $C_1$. Let $\mP$ be the cluster partition defined by Lemma \ref{lem:cluster}. We have the following properties. \vspace{0.1cm}

$(i)$ For all $P\in\mP$ and for all $i,j\in P$ we have that
\begin{equation}\label{eq:holder coyote}
    \forall\,t\in[0,T),\qquad\big|x_i(t)-x_j(t)\big|\;\leq\;C|T-t|^\frac{1}{\alpha+1},
\end{equation} with a constant $C$ that depends only on the $a_i$, $N$, $C_0$, $C_1$ and on $\alpha$. \vspace{0.1cm}

$(ii)$ Assume in addition that $\sum_{i=1}^Na_i\neq0$ and that relation~\eqref{hyp:slow center of vorticity} holds true also for $P = \{1,\ldots,N\}$. Then for all $i=1\dots N$ there exists $x_i^\ast\in\RR^2$ such that
\begin{equation}\label{eq:convergence!}
    x_i(t)\;\longrightarrow\; x_i^\ast,\qquad\text{as }t\to T.
\end{equation}
Moreover,
\begin{equation}\label{eq:holder bip bip}
    \forall\,t\in[0,T),\qquad\big|x_i(t)-x_i^\ast\big|\;\leq\;C|T-t|^\frac{1}{\alpha+1},
\end{equation}
with a constant $C$ that depends only on the $a_i$, $N$, $C_0$, $C_1$and on $\alpha$.\vspace{0.1cm}
\end{lemma}

\begin{proof}
We start with Case $(i)$. Let $P \in \mP$ be the cluster partition defined by Lemma \ref{lem:cluster} and $i,j \in P$. Let $t<T$ and let us define $\eta := |x_i(t)-x_j(t)|$. We consider two cases: $\eta\leq1$ and $\eta>1$.

Assume first that $\eta\leq1$. The collision condition \eqref{eq:pointsCollide}, in view of Proposition \ref{prop:prevent collapse}, gives that $T-t > C_2\, \eta^{\alpha+1}$. Indeed, otherwise Proposition \ref{prop:prevent collapse} would imply that $|x_i(\tau)-x_j(\tau)| \ge \eta/2$
for all $\tau \ge t$. This gives that $x_i$ and $x_j$ do not collide, which contradicts the definition of $\mP$ from Lemma~\ref{lem:cluster} (recall that $i,j \in P \in \mP$). In other words, when $i,j$ belong to the same cluster of collapse $P$, we have that
\begin{equation}
    T-t > C_2\, |x_i(t)-x_j(t)|^{\alpha+1}.
\end{equation}  
This is enough to get~\eqref{eq:holder coyote}. 
It is a direct consequence of Proposition~\ref{prop:prevent collapse} that the constant appearing above depends only on $N$, $a_i$ $C_0$, $C_1$ and $\alpha$. 

We consider now the case $\eta> 1$. Since we are in a cluster of collapse, there is a time $t'>t$ such that the distance between the vortices is equal to $1$, by continuity of the trajectories and~\eqref{eq:pointsCollide}. 
We then apply Proposition \ref{prop:prevent collapse} similarly as above but at time $t'$ to get
\begin{equation}\label{eq:quasi holder 2}
    T-t \geq T-t'\geq C_2\, |x_i(t')-x_j(t')|^{\alpha+1}=C_2.
\end{equation}  
We also know that the relative positions are uniformly bounded by Theorem~\ref{thrm:uniform relative bound}. 
We denote by $C_3$ the constant given by this theorem for intervals of time smaller than $1$ (the constant given by Theorem~\ref{thrm:uniform relative bound} may blow-up as $T\to+\infty$, this is the reason why we assume $T\leq1$). 
More precisely, we make use of a slightly better result than Theorem~\ref{thrm:uniform relative bound}, that is given by Proposition 4.1 in~\cite{Godard-Cadillac_2021}. 
This proposition states in particular that\footnote{Proposition 4.1 in~\cite{Godard-Cadillac_2021} is only stated for the point-vortex system but the only property of the point-vortices that is used is~\eqref{hyp:slow center of vorticity} so that the bound~\eqref{eq:James Bound} remain true in the general case studied here.} there exists a constant $C_3$ depending only on $N$, the intensities $a_i$, $C_0$, and $C_1$ such that
\begin{equation}\label{eq:James Bound}
\forall\,t_1<t_2\in[0,T),\qquad\Big|\big(x_i(t_2)-x_j(t_2)\big)-\big(x_i(t_1)-x_j(t_1)\big)\Big|\leq C_3.
\end{equation}
In the equation above we take $t_1=t$ and we let $t_2\to T$. Using~\eqref{eq:pointsCollide}, this gives $|x_i(t)-x_j(t)|\leq C_3$.
Therefore, with~\eqref{eq:quasi holder 2},
\begin{equation}
    |x_i(t)-x_j(t)|^{\alpha+1}\leq C_3^{\alpha+1}\leq \frac{C_3^{\alpha+1}}{C_2}(T-t).
\end{equation}
This concludes the proof of~\eqref{eq:holder coyote} and then of Case $(i)$. \vspace{1mm}

%%%%%%%%%%%%%%%%%%%%%%%%%%%%

We now go to Case $(ii)$. We start by proving that the trajectories converge as $t \to T$. Since $\sum_{i=1}^N a_i \neq 0$, we have that the non neutral clusters hypothesis \eqref{eq:no null partial sum} holds true and thus we can apply Lemma~\ref{coro:B_control} (see also Remark~\ref{the_remark}) to get that for any $P \in \mP$ and for every $i\in P$,
\begin{equation}
    |x_i(t)-B_P(t)| \le \frac{a}{A}\,\max_{j\in P} |x_i(t)-x_j(t)|.
\end{equation}
This and relation~\eqref{eq:holder coyote} imply that
\begin{equation}\label{eq:x-B_quasi_holder}
        |x_i(t)-B_P(t)| \le \frac{a}{A} C\,|T-t|^\frac{1}{\alpha+1}.
\end{equation}
Finally, by hypothesis, $B_P$ satisfies \eqref{hyp:slow center of vorticity}. 
By Lemma \ref{lem:cluster} we know that relation \eqref{eq:pointsStayFar} holds so that \eqref{hyp:slow center of vorticity} becomes
\begin{equation}\label{BigBrainTime}
    \left| \der{}{t} B_P(t)\right| \leq N^2\frac{C_0}{\delta^\alpha}+C_1.
\end{equation}
This means that $t\mapsto B_P(t)$ itself is uniformly Lipschitz (with a constant depending on $\delta$) and therefore converges as $t\to T$. This fact together with~\eqref{eq:x-B_quasi_holder} proves that $x_i$ is converging towards some $x_i^\ast$. This also gives a rate of convergence however the Hölder constant depends on $\delta$ because of relation \eqref{BigBrainTime}. In order to prove that this constant does not depend on $\delta$, we use an additional argument.

We set $X^\ast:=\{x_i^\ast:i=1\dots N\}$. 
Up to an omitted modification of the labels of the indices for the points $x_i$, there exists $K\leq N$ such that
\begin{equation}X^\ast=\big\{x_i^\ast:i=1\dots K\big\},\qquad\text{and}\qquad\forall\,i\neq j\in\{1,\dots,K\},\quad x_i^\ast\neq x_j^\ast.\end{equation}
We introduce the evolution system $\zeta_k$ with $k=1\dots N+K$ defined by
\begin{equation}
    \zeta_i(t):=x_i(t),\quad\text{for }i=1\dots N\qquad\text{and}\qquad\zeta_{N+i}(t):=x_i^\ast,\quad\text{for }i=1\dots K.
\end{equation}
The intensities $b_k$ associated to this system are defined in the following manner. If $k\leq N$ we set $b_k=a_k$. 
If $k\geq N+1$, the $b_k$ are chosen such that the non-neutral clusters hypothesis~\eqref{eq:no null partial sum} holds true for $(a_i)_{1\le i \le N+K}$. 
Note that such a choice is always possible since we are in the case where the $a_i$ satisfy \eqref{eq:no null partial sum}.
We now compute for $ P\subseteq\{1,\dots, N+K\}$,
\begin{equation}
    \der{}{t}B_{ P}=\bigg(\sum_{k\in P}b_k\bigg)^{-1}\sum_{k\in P}b_k\,\der{\zeta_k}{t}=\bigg(\sum_{k\in P}b_k\bigg)^{-1}\sum_{\substack{i\in P\\i\leq N}}a_i\,\der{x_i}{t}.
\end{equation}

If $P \cap \{1,\ldots,N\} \neq \emptyset$ then $\der{}{t} B_P = 0$. Else $P \cap \{1,\ldots,N\} \neq \emptyset$ and we infer that since the dynamics of the $(x_i)_{1\le i \le N}$ satisfies~\eqref{hyp:slow center of vorticity},
\begin{align*}
    \bigg|\der{}{t}B_{ P}\bigg| & = \Bigg| \bigg(\sum_{k\in P}b_k\bigg)^{-1} \Bigg(\sum_{\substack{i\in P\\i\le N}}a_i \Bigg)\der{}{t} B_{P \cap \{1,\ldots,N\}}\Bigg| \\ & \leq \sum_{\substack{i\in P\\i\leq N}}\;\sum_{\substack{j\notin P\\j\leq N}}\;\frac{C_0'}{|x_i(t)-x_j(t)|^\alpha}+C_1' \\
& \leq \sum_{k\in P}\;\sum_{\ell\notin P}\;\frac{C_0'}{|\zeta_k(t)-\zeta_\ell(t)|^\alpha}+C_1',
\end{align*}
where
\begin{equation}
    C_0':=C_0\Bigg|\bigg(\sum_{k\in P}b_k\bigg)^{-1}\sum_{\substack{i\in P\\i\leq N}}a_i\Bigg|\qquad\text{and}\qquad C_1':=C_1\Bigg|\bigg(\sum_{k\in P}b_k\bigg)^{-1}\sum_{\substack{i\in P\\i\leq N}}a_i\Bigg|.
\end{equation}

Therefore, the system of points $\zeta_k(t)$ satisfies hypothesis~\eqref{hyp:slow center of vorticity}, and then the conclusions of Lemma~\ref{lem:cluster} and of Lemma~\ref{lem:traj are holder}-$(i)$ hold true.
In particular, if we take $i\in\{1,\dots,N\}$ and if we take $k\in\{1,\dots,K\}$ such that $x_i^\ast=x_k^\ast$, we have that $\zeta_i$ and $\zeta_{N+k}$ belong to the same cluster of collapse.
We can then apply Lemma~\ref{lem:traj are holder}-$(i)$ to the dynamics of the $\zeta_k$ and relation~\eqref{eq:holder coyote} for $\zeta_i$ and $\zeta_{N+k}$ to get that
\begin{equation}
\big|x_i(t)-x_i^\ast\big|\leq C |T-t|^\frac{1}{\alpha+1},
\end{equation}
with a constant $C$ that depends only on $N$, $\alpha$, $C_0$, $C_1$ and on the $a_i$. \vspace{1mm}

\end{proof}

%%%%%%%%%%%%%%%%%%%%%%%%%%%%%%%%%%%%%%%%%%%%%%%%%%%%%%%%%%%%%%%%%%%

\subsubsection{End of the proofs for Theorems~\ref{thrm:plane} and~\ref{thrm:weak}}

To end both proofs, we note that thanks to relation~\eqref{eq:quasi-preservation} the point-vortex dynamics \eqref{eq:evo alpha} satisfies \eqref{hyp:slow center of vorticity} with $C_1 = 0$.
We can then apply Lemma~\ref{lem:traj are holder} with $p=2$ (since the point-vortices evolve in the plane) to get the announced Hölder regularity. \vspace{0.2cm}

\noindent\emph{Proof of Theorem~\ref{thrm:plane}}: This is a direct consequence of Lemma~\ref{lem:traj are holder}-$(ii)$ applied on the time interval $[t_1,t_2]$, that is by making the time $t_2$ play the role of the final time. This is possible since the multiplicative constant does not depend on $\delta$.\vspace{0.2cm}

\noindent\emph{Proof of Theorem~\ref{thrm:weak}-}: This case is simply a reformulation Lemma~\ref{lem:traj are holder}-$(i)$. \vspace{0.2cm}

The proofs of Theorems~\ref{thrm:plane} and~\ref{thrm:weak} are completed.\qed\vspace{0.2cm}

We are not able in Theorem~\ref{thrm:weak} to establish the Hölder regularity even of the relative trajectories on the time interval $[0,T)$. This would require to apply Lemma~\ref{lem:traj are holder}-$(i)$ on the time interval $[0,t_2]$. If we do that, since there is no collapse at time $t_2$, the cluster partitions is formed of singletons. Then the constant $\delta$ at time $t_2$ is 
\begin{equation}
    \delta(t_2) = \min_{i\neq j} \min_{t \in [0,t_2]} |x_i(t)-x_j(t)|.
\end{equation}
Since the Hölder constant depends on $1/\delta(t_2)$ and $\delta(t_2) \to 0$ as $t_2 \to T$ (if there is a collapse at time $T$), this Hölder constant can not be chosen uniformly in $t_2$.

%%%%%%%%%%%%%%%%%%%%%%%%%%%%%%%%%%%%%%%%%%%%%%%%%%%%%%%%%%%%%%%%%%%
%%%%%%%%%%%%%%%%%%%%%%%%%%%%%%%%%%%%%%%%%%%%%%%%%%%%%%%%%%%%%%%%%%%
%%%%%%%%%%%%%%%%%%%%%%%%%%%%%%%%%%%%%%%%%%%%%%%%%%%%%%%%%%%%%%%%%%%
\section{Proof of Hölder regularity in bounded domains: Theorem~\ref{thrm:bounded}.}
In this section we are interested in the dynamics of vortices that lay in $\Omega$, a smooth bounded domain. We restrict our study to the case $\alpha = 1$. We assume for all this section that the intensities satisfy the non neutral cluster hypothesis~\eqref{eq:no null partial sum}.

The proof of the convergence property for the vortices in $\Omega$ is inspired from an analogous previous result for the point-vortex model in the plane~\cite{Godard-Cadillac_2021}. 
The Hölder regularity makes use of Proposition~\ref{prop:prevent collapse}. For the vortices going to the boundary in finite time, we will first prove convergence of the distance to the boundary, then establish the Hölder regularity with the use of Proposition~\ref{prop:prevent collapse_v2}, similar to Proposition~\ref{prop:prevent collapse}.

%%%%%%%%%%%%%%%%%%%%%%%%%%%%%%%%%%%%%%%%%%%%%%%%%%%%%%%%%%%%%%%%%%%
%%%%%%%%%%%%%%%%%%%%%%%%%%%%%%%%%%%%%%%%%%%%%%%%%%%%%%%%%%%%%%%%%%%

Let $(t \mapsto x_i(t))_i$ be a solution of the Euler point-vortex problem~\eqref{eq:evolution euler non simply connected} in $\Omega$ and assume that there is no collapse on the time interval $[0,T).$
For all $t\in[0,T)$ define the distribution $P_t\in\cD'(\Omega)$ by
\begin{equation}\label{def:P_t}
    P_t:=\sum_{i=1}^Na_i\,\delta_{x_i(t)}.
\end{equation}

The proof of Theorem~\ref{thrm:bounded} is separated into three parts. 
We first study the distribution $P_t$ for all $t\in[0,T)$ and prove that it converges as $t\to T$. 
Secondly, we use this result to obtain the convergence of the position of the vortices that have an accumulation point inside $\Omega$.
In the last part, we use this convergence result to prove the announced Hölder regularity results with arguments similar to the full plane case (Theorem~\ref{thrm:plane}).

%%%%%%%%%%%%%%%%%%%%%%%%%%%%%%%%%%%%%%%%%%%%%%%%%%%%%%%%%%%%%%%%%%%
\subsection{Evolution in time of the studied distribution}

\begin{lemma}\label{lem:distri speed}
Let $\varphi$ be $\cC^\infty$ and compactly supported inside $\Omega$ and let $P_t$ be the distribution on $\Omega$ defined by~\eqref{def:P_t}. Then,
\begin{equation}
    \bigg|\der{}{t}\left<P_t,\varphi\right>\bigg|\;\leq\;C\|\nabla\varphi\|_{L^\infty}\bigg(\frac{1}{\dist(\supp(\varphi),\partial\Omega)}+1\bigg) + C\,\|\nabla^2\varphi\|_{L^\infty},
\end{equation}
where $supp(\varphi)$ is the support of $\varphi$ and where
\begin{equation}
    \dist(A,B):=\inf_{x\in A}\;\inf_{y\in B}\;|x-y|.
\end{equation}
The constant $C$ depends on the coefficients $a_i$, on the circulations $\xi_m$ and on $\Omega$.
\end{lemma}
\begin{proof}
A direct computation using the evolution equations for the point-vortex problem in bounded domains under developed form~\eqref{eq:evolution euler reformulate} gives
\begin{equation}\label{canard}\begin{split}
    &\der{}{t}\left<P_t,\varphi\right>=\der{}{t}\sum_{i=1}^Na_i\varphi\big(x_i(t)\big)=\sum_{i=1}^Na_i\nabla\varphi\big(x_i(t)\big)\cdot\der{}{t}x_i(t)\\
    &=\sum_{i=1}^N\sum_{j=1}^Na_i\,a_j\nabla\varphi\big(x_i(t)\big)\cdot\nabla^\perp_x\gamma_\Omega\big(x_i(t),x_j(t)\big)\\
    &\qquad+\frac{1}{2\pi}\sum_{i=1}^N\sum_{\substack{j=1\\j\neq i}}^Na_i\,a_j\,\nabla\varphi\big(x_i(t)\big)\cdot\nabla^\perp\,G_1\big(x_i(t)-x_j(t)\big)\\
    &\qquad+\sum_{i=1}^Na_i\nabla\varphi\big(x_i(t)\big)\cdot\Big(\sum_{m=1}^Mc_m(t)\beta_m\big(x_i(t)\big)\Big)\\
    &:=S_1+S_2+S_3.
    \end{split}
\end{equation}
The first sum appearing on the right-hand side of~\eqref{canard}, noted $S_1$, is estimated term by term using~\eqref{eq:kernel robin estimate}:
\begin{equation}\label{oie}
    \forall\;i=1\dots N,\qquad\Big|\nabla\varphi\big(x_i(t)\big)\cdot\nabla^\perp_x\gamma_\Omega\big(x_i(t),x_j(t)\big)\Big|\;\leq\;C\frac{\big|\nabla\varphi\big(x_i(t)\big)\big|}{\dist\big(x_i(t),\partial\Omega\big)}.
\end{equation}
One continues the estimate of~\eqref{oie} using that the function $\varphi$ is compactly supported inside $\Omega$:
\begin{equation}
\Big|\nabla\varphi\big(x_i(t)\big)\cdot\nabla^\perp_x\gamma_\Omega\big(x_i(t),x_i(t)\big)\Big|\;\leq\;C\left\{
     \begin{array}{ll}
     \displaystyle\frac{\|\nabla\varphi\|_{L^\infty}}{\dist\big(x_i(t),\partial\Omega\big)}&\qquad\text{if }x_i(t)\in\supp(\varphi),\vspace{0.2cm}\\
     0&\qquad\text{otherwise.}\end{array}\right.
\end{equation}
Thus,
\begin{equation}
    |S_1|\;\leq\;C\,\frac{\|\nabla\varphi\|_{L^\infty}}{\dist(\supp(\varphi),\partial\Omega)}.
\end{equation}
To estimate the second term in~\eqref{canard}, denoted by $S_2$, one proceeds to a symmetrization of the double sum by swapping the indices $i\leftrightarrow j$, in the spirit of Delort and Schochet~\cite{Delort_1991, Schochet_1996}. This gives
\begin{equation}
    S_2=\frac{1}{4\pi}\sum_{i=1}^N\sum_{\substack{j=1\\j\neq i}}^Na_i\,a_j\,\Big(\nabla\varphi\big(x_i(t)\big)-\nabla\varphi\big(x_j(t)\big)\Big)\cdot\nabla^\perp G_1\big(x_i(t)-x_j(t)\big).
\end{equation}
Using relation~\eqref{eq:K alpha},
\begin{equation}\begin{split}
    \bigg|\Big(\nabla\varphi\big(x_i(t)\big)-\nabla\varphi\big(x_j(t)\big)\Big)\cdot\nabla^\perp G_1\big(x_i(t)-x_j(t)\big)\bigg|\\
    \leq\frac{\big|\nabla\varphi\big(x_i(t)\big)-\nabla\varphi\big(x_j(t)\big)\big|}{|x_i(t)-x_j(t)|}\leq\|\nabla^2\varphi\|_{L^\infty},
    \end{split}
\end{equation}
where the last inequality is simply the mean value theorem applied to the function $\nabla\varphi.$ Thus, $|S_2|\leq C\|\nabla^2\varphi\|_{L^\infty}$. Concerning the last term $S_3$, since the functions $\beta_m$ and $c_m$ are bounded by standard elliptic estimates, we have that $|S_3|\leq C\|\nabla\varphi\|_{L^\infty}.$ Gathering these three estimates back into~\eqref{canard} concludes the proof.
\end{proof}

Now that we obtained an upper bound on the derivative of this distribution we study its limit as $t\to T.$

\begin{lemma}\label{lem:distri converge}
Let $(t \mapsto x_i(t))_i$ be a solution of the Euler point-vortex problem in bounded domains~\eqref{eq:evolution euler non simply connected}. 
Define the distribution $P_t$ with~\eqref{def:P_t}. Then for all $i=1\dots N$, there exist $x_i^\ast\in\overline{\Omega}$ such that 
\begin{equation}
    P_t\longrightarrow\sum_{i=1}^Na_i\,b_i\,\delta_{x_i^\ast},\qquad\text{as }t\to T,\quad\text{in the weak sense of measures on $\Omega$.}
\end{equation}
where
\begin{equation}\label{def:b_i}
b_i:=\left\{\begin{array}{ll}0&\quad\text{ if }x_i^\ast\in\partial\Omega,\\ 1&\quad\text{ if } x_i^\ast\in\Omega.\end{array}\right.
\end{equation}
\end{lemma}
\begin{proof}
First, as a consequence of Lemma~\ref{lem:distri speed}, the derivative in time of $\left<P_t,\varphi\right>$ is bounded for all fixed choice of $\varphi \in\cC^\infty,$ compactly supported in $\Omega$. This implies that $\left<P_t,\varphi\right>$ converges as $t\to T$ for all $\varphi.$
Therefore, there exists a distribution $P^\ast\in\cD'(\Omega)$ such that
\begin{equation}
        P_t\longrightarrow P^\ast\qquad\text{as }t\to T,\quad\text{in the distributional sense.}
\end{equation}

On the other hand, since $\Omega$ is a bounded subset of $\RR^2$, the vector $\big(x_i(t)\big)_{i=1}^N$ is bounded in $\overline{\Omega}^N$ as $t\to T$.
Therefore, by compactness, there exists an increasing sequence of time $(t_n)$ converging towards $T$ and a family of $N$ points $x_i^\ast \in \overline{\Omega}$ such that for all $i=1\dots N$,
\begin{equation}\label{coq}
x_i(t_n)\longrightarrow x_i^\ast.
\end{equation} 
%%%%%%%%%%%%%%%%%%%%%%%%%%
As a consequence of~\eqref{def:b_i} and~\eqref{coq}, since we have for any $C^\infty$ map $\varphi$ compactly supported in $\Omega$ that
\begin{equation}
\left<\sum_{i=1}^Na_i\,\delta_{x_i(t_n)},\varphi\right> = \sum_{i=1}^N a_i\varphi(x_i(t_n)) \longrightarrow  \sum_{i=1}^Na_i b_i\varphi (x_i^\ast) = \left<\sum_{i=1}^Na_i\,b_i\,\delta_{x_i^\ast},\varphi\right>\qquad\text{as}\;n\to+\infty,
\end{equation}
we have the following convergence in the distributional sense:
\begin{equation}
\sum_{i=1}^Na_i\,\delta_{x_i(t_n)}\longrightarrow\sum_{i=1}^Na_i\,b_i\,\delta_{x_i^\ast}\qquad\text{as}\;n\to+\infty.
\end{equation}
By uniqueness of the limit, it is possible to identify
\begin{equation}
P^\ast = \sum_{i=1}^Na_i\,b_i\,\delta_{x_i^\ast}.
\end{equation}
The fact that the convergence of $P_t$ towards $P^\ast$ in $\cD'$ is 
actually a convergence in the weak sense of measures comes from the 
fact that the measure $P_t$ is bounded (by $\sum_i|a_i|$) for all $t$.
\end{proof}

%%%%%%%%%%%%%%%%%%%%%%%%%%%%%%%%%%%%%%%%%%%%%%%%%%%%%%%%%%%%%%%%%%%
\subsection{Convergence of the vortices}

\begin{lemma}[Accumulation points for the vortices]\label{lem:adherence set}
Let $(t \mapsto x_i(t))_i$ be a solution of the Euler point-vortex problem~\eqref{eq:evolution euler non simply connected} in $\Omega$. 
Consider a set of points $x_i^\ast$ given by Lemma~\ref{lem:distri converge}. 

Then, for all $i=1\dots N,$
\begin{equation}
    \Big\{x_0\in\overline{\Omega}\;:\;\liminf_{t\to T}|x_i(t)-x_0|=0\Big\}\;\subseteq\;\partial\Omega\cup\Big(\bigcup_{k=1}^N\{x_k^\ast\}\Big).
\end{equation}
\end{lemma}

\begin{proof}
Let $x_0\in \Omega$. Assume that there exists an index $i$ and a sequence of times $(t_n)$ such that
\begin{equation}
    x_i(t_n)\longrightarrow x_0,\qquad\text{as }t\to T.
\end{equation}
Since $\Omega$ is bounded, one can assume, up to an omitted extraction, that for all $j=1\dots N$, there exists $x_j^\dag\in\overline{\Omega}$ such that
\begin{equation}
       x_j(t_n)\longrightarrow x_j^\dag,\qquad\text{as }t\to T.
\end{equation}
As in the proof of Lemma~\ref{lem:distri converge}, we have that
\begin{equation}
\sum_{j=1}^Na_j\,\delta_{x_j(t_n)}\longrightarrow\sum_{j=1}^Na_j\,b'_j\,\delta_{x_j^\dagger}\qquad\text{as}\;n\to+\infty,
\end{equation}
where the $b'_j$ are defined as in \eqref{def:b_i} with the points $x_j^\dagger$.

By uniqueness of the limit, recalling that by Lemma~\ref{lem:distri converge}, $P_t \to P^\ast$, we have that
\begin{equation}\label{eq:identification des mesures}
    \sum_{j=1}^Na_j\,b'_j\,\delta_{x_j^\dagger} = \sum_{j=1}^Na_j\,b_j\,\delta_{x_j^\ast}.
\end{equation}
We focus on the left-hand side of this equation. We want to prove that the coefficient of $\delta_{x_0}$ is not 0. Let $J = \big\{j = 1 \dots N : x_j^\dagger = x_0\big\}$. Then, since $x_0 \in \Omega$, this coefficient is
\begin{equation}
    \sum_{j\in J} a_j b'_j = \sum_{j\in J} a_j.
\end{equation}
By definition of $J$, $i \in J$, thus $J$ is non empty. By the non neutral cluster hypothesis \eqref{eq:no null partial sum}, the coefficient of $\delta_{x_0}$ is not zero. Therefore relation \eqref{eq:identification des mesures} gives that $x_0\in\{ x_1^\ast,\ldots,x_N^*\}$.
\end{proof}

Now that we have a better description of the possible accumulation points for the vortex $x_i(t)$, it remains to study the convergence of its trajectory.

\begin{lemma}[The \emph{lim inf} is actually a limit]\label{lem:liminf implies lim}
Let $(t \mapsto x_i(t))_i$ be a solution of the Euler point-vortex problem~\eqref{eq:evolution euler non simply connected} in $\Omega$. 

If $x_0\in\Omega$ is such that
\begin{equation}\label{eq:liminf hypothesis}
    \liminf_{t\to T^-}\big|x_i(t)-x_0\big|=0,
\end{equation}
then
\begin{equation}
    \lim_{t\to T^-}\big|x_i(t)-x_0\big|=0.
\end{equation}
\end{lemma}
\begin{proof}
We apply Lemma \ref{lem:distri converge} to construct the points $x_k^\ast\in\Omega$. As a consequence of Lemma~\ref{lem:adherence set}, there exists an index $k$ such that $x_0=x_{k}^\ast$.
Assume for the sake of contradiction that there exists $x^\dag\neq x_0$ in $\overline{\Omega}$ such that
\begin{equation}\label{saumon}
    \liminf_{t\to T^-}\big|x_i(t)-x^\dag\big|=0.
\end{equation}
The two \emph{lim inf} given by~\eqref{eq:liminf hypothesis} and by~\eqref{saumon} imply the existence of an increasing sequence of time $(t_n)$ converging towards $T$ such that
\begin{equation}\label{truite}
        \lim_{n\to+\infty}\big|x_i(t_{2n})-x^\ast_k\big|=0\qquad\text{and}\qquad    \lim_{n\to+\infty}\big|x_i(t_{2n+1})-x^\dag\big|=0.
\end{equation}
Let us define
\begin{equation}
    J:=\big\{j=1\dots N:x_k^\ast\neq x_j^\ast\big\}
\end{equation}
and
\begin{equation}\begin{split}
    &\delta_0:=\dist(x_k^\ast,\partial\Omega)>0,\\
    &\delta^\ast:=\min_{j\in J}|x_k^\ast-x_j^\ast|\in(0,+\infty],\\
    &\delta^\dag:=|x_k^\ast-x^\dag|>0,\\
    &\delta:=\frac{1}{2}\min\big\{\delta_0;\delta^\ast;\delta^\dag\big\}>0.
\end{split}
\end{equation}
With such definitions and since $\delta<\delta^\dag$,
\begin{equation}
    x^\dag\notin \cB\big(x_k^\ast,\delta\big).
\end{equation}
Then, by continuity of the trajectories and using the intermediate value theorem, Condition~\eqref{truite} implies (when $n$ is large enough) the existence of a time $\tau_n\in[t_{2n};t_{2n+1}]$ such that
\begin{equation}
    \big|x_i(\tau_n)-x_k^\ast\big|=\delta.
\end{equation}
In other words, $x_i(\tau_n)$ belongs to the circle of center $x_k^\ast$ and of radius $\delta$, noted $\cC(x_k^\ast,\delta)$. This circle being a compact set, up to an omitted extraction the following convergence holds:
\begin{equation}
    x_i(\tau_n)\;\longrightarrow\;\widehat{x_i}\in\cC(x_k^\ast,\delta).
\end{equation}
This point $\widehat{x_i}$ is an accumulation point of $x_i(t)$ as $t\to T$.
By definition of $\delta$, one can easily check that $\cC(x_k^\ast,\delta)\subseteq\Omega \setminus \{ x_1^\ast,\ldots,x_N^\ast \}$. We conclude that $\widehat{x_i}$ is an accumulation point of $x_i(t)$ as $t\to T$ which belongs to $\Omega$ but not to $\{ x_1^\ast,\ldots,x_N^\ast\}$. 
This is in contradiction with Lemma~\ref{lem:adherence set}.
\end{proof}

%%%%%%%%%%%%%%%%%%%%%%%%%%%%%%%%%%%%%%%%%%%%%%%%%%%%%%%%%%%%%%%%%%%
%%%%%%%%%%%%%%%%%%%%%%%%%%%%%%%%%%%%%%%%%%%%%%%%%%%%%%%%%%%%%%%%%%%
\subsection{Hölder regularity properties}
There remain to establish the announced Hölder regularity for the trajectories that converge inside $\Omega$ as $t\to T$ and the Hölder regularity of the distance with the boundary for the vortices that collapses with $\partial\Omega$.

%%%%%%%%%%%%%%%%%%%%%%%%%%%%%%%%%%%%%%%%%%%%%%%%%%%%%%%%%%%%%%%%%%%
\subsubsection{Choosing the right interval of time}
To start with, we recall that the boundary of $\Omega$ is assumed to be smooth enough to have its curvature well-defined and bounded. This implies the existence of a $\overline{\delta}>0$ such that for all $x\in\Omega$ such that $\dist(x,\partial\Omega)\leq\overline{\delta}$, there exists a unique $x'\in\partial\Omega$ such that
\begin{equation}
    |x-x'|=\dist(x,\partial\Omega).
\end{equation}
In other words, the projection of $x$ on $\partial\Omega$ is well-defined on the set $\Omega\cap$ $\!(\partial\Omega+\cB(0,\overline{\delta}))$.
This projection is denoted by $P_{\partial\Omega}$. 
If $\overline{\delta}$ is chosen small enough (or, which is equivalent, if we replace $\overline{\delta}$ by $\overline{\delta}/2$), then the function $P_{\partial\Omega}$ is smooth and the gradient of the distance between $x$ and $\partial\Omega$ is given by:
\begin{equation}\label{eq:gradient_distance}
    \nabla_x\dist(x,\partial\Omega)=\frac{x-P_{\partial\Omega}x}{|x-P_{\partial\Omega}x|}.
\end{equation}
See for instance \cite[Appendix 14.6]{Gilbarg_Trudinger_2001_elliptic},  for details.
Moreover, it satisfies the following Lipschitz estimate: for all $x,y\in\Omega\cap$ $\!(\partial\Omega+\cB(0,\overline{\delta}))$,
\begin{equation}
\big|P_{\partial\Omega}x-P_{\partial\Omega}y\big|\,\leq\, C_{\partial\Omega}\,|x-y|,
\end{equation}
where $C_\Omega$ is a constant that depends only on the maximal curvature of $\partial\Omega$. Note that $\overline{\delta}$ also depends only of the maximal curvature of $\partial\Omega$.

Recall that we defined by~\eqref{def:I}:
\begin{equation}\label{def:I again}
    I:=\big\{i=1\dots N\,:\,\liminf\limits_{t\to T^-}\;\dist\big(x_i(t),\partial\Omega\big)=0\big\}.
\end{equation}
If $i \notin I$, by compactness, the trajectory $x_i(t)$ has an accumulation point inside $\Omega$ as $t \to T^-$. By Lemma~\ref{lem:liminf implies lim} this means that
\begin{equation}\label{eq:cv_inside}
    \forall\;i\notin I,\quad\exists\;x_i^\ast\in\Omega,\qquad x_i(t)\longrightarrow x_i^\ast\quad\text{as }t\to T^-.
\end{equation}
If $i \in I$, then by Lemma~\ref{lem:liminf implies lim}, it is not possible for $x_i(t)$ to have an accumulation point inside $\Omega$ as $t \to T^-$ and thus 
\begin{equation}\label{eq:cv_bord}
    \forall i \in I, \qquad \lim_{t\to T^-}\dist(x_i(t),\partial\Omega) = 0.
\end{equation}

We define $d_0$ as
\begin{equation}\label{def:d_0}
     d_0 :=\min_{i\notin I}\inf_{t \in [0,T)}\dist(x_i(t),\partial\Omega).
\end{equation}
Clearly $d_0 > 0$. 
We also define
\begin{equation}\label{def:delta again}
   \delta:=\frac{1}{4}\min\{\overline{\delta};d_0;1\}.
\end{equation}

From those definitions and relations \eqref{eq:cv_inside} and \eqref{eq:cv_bord}, we get the existence of a time $T_\delta\in[0,T)$ depending on $\delta$ and thus on $d_0$ and $\Omega$ such that for all $t\in[T_\delta;T)$:
\begin{equation}\label{eq:separation of converging vortices}
    \forall\;i\notin I,\quad\dist(x_i(t),\partial\Omega)\geq\frac{3}{4}\delta\qquad\text{and}\qquad\forall\;i\in I,\quad\dist(x_i(t),\partial\Omega)\leq\frac{1}{4}\delta.
\end{equation}
More precisely, we define $T_\delta$ as follows:
\begin{equation}\label{def:T delta}
    T_\delta:=\inf\bigg\{\tau\in[0,T)\,:\,\text{condition}~\eqref{eq:separation of converging vortices}~\text{holds true for all } t \in [\tau,T)\bigg\}<T.
\end{equation}
We are going to establish the Hölder regularity properties near the collapse on the time interval $[T_\delta, T)$. 

%%%%%%%%%%%%%%%%%%%%%%%%%%%%%%%%%%%%%%%%%%%%%%%%%%%%%%%%%%%%%%%%%%%
\subsubsection{Hölder regularity property for the vortices far from the boundary}
The vortices $x_i$ with $i\notin I$ are the vortices that remain far from the boundaries of $\Omega$. Let $i\notin I$, the equation of evolution for such a vortex is given by
\begin{equation}\label{eq:evolution restricted}
\der{x_i(t)}{t} = \frac{1}{2\pi}\sum_{\substack{k \notin I\\k\neq i}}a_k\nabla^\perp\,G_1\big(x_i(t)-x_k(t)\big) + f_i(t),
\end{equation}
where $f_i$ denotes the influence of the boundary of $\Omega$ and of the vortices that are colliding with the boundary:
\begin{equation}
    f_i(t) = \sum_{k=1}^N a_k\nabla^\perp_x\gamma_\Omega\big(x_i(t),x_k(t)\big)+\frac{1}{2\pi}\sum_{k \in I} a_k\nabla^\perp\,G_1\big(x_i(t)-x_k(t)\big)+\sum_{m=1}^Mc_m(t)\beta_m\big(x_i(t)\big).
\end{equation}
On the interval of time $[T_\delta, T)$, it is a consequence of~\eqref{eq:kernel robin estimate} and~\eqref{eq:separation of converging vortices} that the function $f_i$ is bounded by a constant that depends on $a_i$, on $N$, on $\Omega$, on $\delta>0$ and on $\xi_m$. The dependency with respect to $\xi_m$ comes from the definition of $c_m(t)$ at~\eqref{def:c_m}.
Therefore, the dynamics \eqref{eq:evolution restricted} satisfies \eqref{hyp:slow center of vorticity}, since for any $P \subset \{1,\dots,N\} \setminus I$, 
\begin{align*}
    \left|\der{}{t} B_P\right| & = \bigg| \sum_{i\in P} a_i\bigg|^{-1} \Bigg|\frac{1}{2\pi}\sum_{i\in P} \sum_{\substack{k \notin I\\k\neq i}}a_ia_k\nabla^\perp\,G_1\big(x_i(t)-x_k(t)\big) + \sum_{i\in I} a_if_i(t)  \Bigg| \\ 
    & \le \sum_{i\in P}\sum_{k \notin P\cup I} \frac{C_0}{|x_i(t)-x_k(t)|} + C_1.
\end{align*} 
If we now consider $t_1<t_2\in[T_\delta, T)$ it is then possible to apply Lemma~\ref{lem:traj are holder}-$(ii)$ to the dynamics~\eqref{eq:evolution restricted} on the interval $[t_1,t_2)$ with $\alpha=1$ to obtain the following Hölder estimate:
\begin{equation}
    \forall\;t_1<t_2\in[T_\delta, T),\qquad\big|x_i(t_2)-x_i(t_1)\big|\,\leq\,C\sqrt{t_2-t_1\,}.
\end{equation}
The constant $C$ depends on $a_i$, on $N$, on $\delta>0$ and on $\xi_m$. To conclude the proof of Theorem~\ref{thrm:bounded}-$(i)$, we observe that on the time interval $[0,T_\delta]$ the trajectories are actually $C^\infty$.

%%%%%%%%%%%%%%%%%%%%%%%%%%%%%%%%%%%%%%%%%%%%%%%%%%%%%%%%%%%%%%%%%%%
\subsubsection{Estimate of the distances to the boundary}
We now study the system of equations for the distance to the boundary of the point-vortices that collapse with the boundary: $t\mapsto\dist(x_i(t),\partial\Omega)$ with $i\in I$.
\begin{lemma}\label{lem:centers of vorticity for the distances}
Consider the point-vortex dynamics $x_i(t)$ with intensities $a_i\neq0$ on bounded multi-connected domains~\eqref{eq:evolution euler non simply connected} and under non-neutral clusters hypothesis~\eqref{eq:no null partial sum}. Define the set $I\subseteq\{1,\dots,N\}$ by~\eqref{def:I again}, the distance $\delta>0$ by~\eqref{def:delta again} and the time $T_\delta$ by~\eqref{def:T delta}. 

Then, the following estimate holds for all $J\subseteq I$ non-empty:
\begin{equation}\begin{split}
    \forall\;t\in[T_\delta, T),\quad&\quad\bigg|\der{}{t}\bigg(\sum_{i\in J}a_i\bigg)^{-1}\sum_{i\in J}a_i\,\dist\big(x_i(t),\partial\Omega\big)\bigg|\\
    &\leq \sum_{i\in J}\sum_{j\in I\setminus J}\frac{C_0}{|\dist(x_i(t),\partial\Omega)-\dist(x_j(t),\partial\Omega)|}+\sum_{i\in J}\frac{C_1}{\dist(x_i(t),\partial\Omega)}+C_2,
\end{split}
\end{equation}
where $C_0$, $C_1$ and $C_2$ are constant depending on $N$, $a_i$, $\Omega$, $\xi_m$ and $\delta$.
\end{lemma}
\begin{proof}
Note first that $P_{\partial\Omega}$ is well-defined for $x_i(t)$ when $i\in I$ and $t\in[T_\delta, T)$ because of~\eqref{eq:separation of converging vortices}. Using~\eqref{eq:evolution euler non simply connected}:
\begin{equation}~\label{eq:evolution distance}\begin{split}
    \der{}{t}\dist(x_i(t),\partial\Omega)&=\nabla_x\dist(x_i(t),\partial\Omega)\cdot\der{}{t}x_i(t) \\
    & = \frac{x_i(t)-P_{\partial\Omega}x_i(t)}{|x_i(t)-P_{\partial\Omega}x_i(t)|}\cdot \der{}{t}x_i(t)\\
    &=g_i(t)+h_i(t)-\ell_i(t)-\frac{x_i(t)-P_{\partial\Omega}x_i(t)}{|x_i(t)-P_{\partial\Omega}x_i(t)|}\cdot\sum_{\substack{j\in I\\j\neq i}}a_j\nabla_x^\perp\cG_\Omega\big(x_i(t),x_j(t)\big)
\end{split}
\end{equation}
where
\begin{equation}
g_i(t):=a_i\frac{x_i(t)-P_{\partial\Omega}x_i(t)}{|x_i(t)-P_{\partial\Omega}x_i(t)|}\cdot\nabla^\perp\gamma_\Omega\big(x_i(t),x_i(t)\big),
\end{equation}
\begin{equation}h_i(t):=\frac{x_i(t)-P_{\partial\Omega}x_i(t)}{|x_i(t)-P_{\partial\Omega}x_i(t)|}\cdot\sum_{m=1}^Mc_m(t)\beta_m\big(x_i(t)\big)
\end{equation}
and
\begin{equation}
    \ell_i(t):=\frac{x_i(t)-P_{\partial\Omega}x_i(t)}{|x_i(t)-P_{\partial\Omega}x_i(t)|}\cdot\sum_{j\notin I}a_j\nabla_x^\perp\cG_\Omega\big(x_i(t),x_j(t)\big).
\end{equation}
We have that the function $g_i(t)$ is bounded by a constant that depends only on $\Omega$, see the proof of Corollary 3.6 in~\cite{Donati_2021}. Concerning $h_i(t)$, the definitions of $c_m(t)$ and $\beta_m$ respectively at~\eqref{def:c_m} and~\eqref{def:beta_m} give that this term is bounded by standard elliptic estimates. The bounds depend on $a_i$, on $N$, on $\Omega$ and on $\xi_m$. It is a consequence of the following estimate
\begin{equation}\label{eq:estimate cG bis}
    \forall x,y \in \Omega, \qquad \big|\nabla_x \cG_\Omega(x,y)\big|\,\leq\,\frac{C_\Omega}{|x-y|},
\end{equation}
obtained from relations~\eqref{def:robin function} and~\eqref{eq:kernel robin general estimate}, that the function $t\mapsto \ell_i(t)$ is bounded by a constant that depends on $a_i$, $\Omega$, $N$ and on $\delta>0$. The remaining term in~\eqref{eq:evolution distance} is the only singular term. 

Let us compute now the evolution of the barycenters for the distances to $\partial\Omega$. Let $J\subseteq I$. Recalling \eqref{def:robin function}, we have that
\begin{equation}\label{eq:evolution distance bary}\begin{split}
    &\qquad\der{}{t}\sum_{i\in J}a_i\,\dist\big(x_i(t),\partial\Omega\big)\\ & = \sum_{i\in J}a_i\Big(g_i(t)+h_i(t)-\ell_i(t)\Big)- \sum_{i\in J} \sum_{\substack{ j\in I\\j\neq i}}a_i\,a_j\frac{x_i(t)-P_{\partial\Omega}x_i(t)}{|x_i(t)-P_{\partial\Omega}x_i(t)|}\cdot\nabla_x^\perp\cG_\Omega\big(x_i(t),x_j(t)\big) \\
    &=\sum_{i\in J}a_i\,\Big(g_i(t)+h_i(t)-\ell_i(t)\Big)+Q_J(t)+R_J(t)-S_J(t)
\end{split}
\end{equation}
where\begin{equation}\label{def:Q_J}
    Q_J(t):=\frac{1}{2\pi}\sum_{i\in J}\sum_{\substack{j\in J\\j\neq i}}a_i\,a_j\frac{x_i(t)-P_{\partial\Omega}x_i(t)}{|x_i(t)-P_{\partial\Omega}x_i(t)|}\cdot\nabla^\perp G_1\big(x_i(t)-x_j(t)\big),
\end{equation}
\begin{equation}\label{def:R_J}
    R_J(t):=\sum_{i\in J}\sum_{\substack{j\in J\\j\neq i}}a_i\,a_j\frac{x_i(t)-P_{\partial\Omega}x_i(t)}{|x_i(t)-P_{\partial\Omega}x_i(t)|}\cdot\nabla_x^\perp\gamma_\Omega\big(x_i(t),x_j(t)\big)
\end{equation}
and
\begin{equation}
    S_J(t):=\sum_{i\in J}\sum_{j\in I\setminus J}a_i\,a_j\frac{x_i(t)-P_{\partial\Omega}x_i(t)}{|x_i(t)-P_{\partial\Omega}x_i(t)|}\cdot\nabla_x^\perp\cG_\Omega\big(x_i(t),x_j(t)\big).
\end{equation}

The most singular term above is \emph{a priori} the term $Q_J$ in \eqref{def:Q_J}. Nevertheless, it can be rewritten using the symmetry property $G_1(x-y)=G_1(y-x)$ to symmetrize the double sum:
\begin{equation}\label{atchoum}
    \begin{split}
    &\sum_{i\in J}\sum_{\substack{j\in J\\j\neq i}}a_i\,a_j\frac{x_i(t)-P_{\partial\Omega}x_i(t)}{|x_i(t)-P_{\partial\Omega}x_i(t)|}\cdot\nabla^\perp G_1\big(x_i(t)-x_j(t)\big)\\
    &=\frac{1}{2}\sum_{i\in J}\sum_{\substack{j\in J\\j\neq i}}a_i\,a_j\bigg(\frac{x_i(t)-P_{\partial\Omega}x_i(t)}{|x_i(t)-P_{\partial\Omega}x_i(t)|}-\frac{x_j(t)-P_{\partial\Omega}x_j(t)}{|x_j(t)-P_{\partial\Omega}x_j(t)|}\bigg)\cdot\nabla^\perp G_1\big(x_i(t)-x_j(t)\big).
    \end{split} 
\end{equation}
Since we only focus on points that are close to the boundary, the gradient of $x\mapsto \dist(x,\partial\Omega)$ is a smooth function. Recalling \eqref{eq:gradient_distance} we have
\begin{equation}
    \bigg|\frac{x_i(t)-P_{\partial\Omega}x_i(t)}{|x_i(t)-P_{\partial\Omega}x_i(t)|}-\frac{x_j(t)-P_{\partial\Omega}x_j(t)}{|x_j(t)-P_{\partial\Omega}x_j(t)|}\bigg|\;\leq C_\Omega\big|x_i(t)-x_j(t)\big|.
\end{equation}
Combining this estimate with $|\nabla G_1(x)|\leq 1/|x|$ gives that the right hand side of~\eqref{atchoum} is bounded by a constant that depends only on $N$, $a_i$ and $\Omega$.
Concerning $R_J$ in \eqref{def:R_J}, we make use of Property~\eqref{eq:kernel robin estimate} on the $\gamma_\Omega$ function to write
\begin{equation}
    |R_J(t)|\leq C\sum_{i\in J}\frac{1}{\dist(x_i(t),\partial\Omega)}.
\end{equation}
where the constant $C$ here depends only on $N$, $a_i$ and $\Omega$.
For the term $S_J(t)$, we first remark that since the distance to the boundary is smooth we have that
\begin{equation}
    \big|\dist(x,\partial\Omega)-\dist(y,\partial\Omega)\big|\leq C_\Omega\,|x-y|
\end{equation}
Plugging this back into~\eqref{eq:estimate cG bis} leads to
\begin{equation}
        \big|\nabla_x\cG_\Omega(x,y)\big|\,\leq\,\frac{C'_\Omega}{|\dist(x,\partial\Omega)-\dist(y,\partial\Omega)|},
\end{equation}
so that,
\begin{equation}
    |S_J(t)|\leq C\sum_{i\in I}\sum_{j\in I\setminus J}\frac{1}{|\dist(x_i(t),\partial\Omega)-\dist(x_j(t),\partial\Omega)|}
\end{equation}
and the constant $C$ depends on $N$, $a_i$ and $\Omega$. Plugging all these estimates back into~\eqref{eq:evolution distance bary} gives the required conclusion.
\end{proof}

%%%%%%%%%%%%%%%%%%%%%%%%%%%%%%%%%%%%%%%%%%%%%%%%%%%%%%%%%%%%%%%%%%%
\subsubsection{Hölder regularity property for the vortices collapsing with the boundary}
With Lemma~\ref{lem:centers of vorticity for the distances} at hand, it is possible to continue the proof of the Hölder regularity with arguments similar to the case of the whole plane. For that purpose, we prove:

\begin{proposition}[Sufficient condition to prevent a collapse - \emph{bis}]\label{prop:prevent collapse_v2}
For $i = 1\dots N$, let $t\mapsto \zeta_i(t)$ be a family of $N$ different points of $\RR^p$ evolving on a time interval $[0,T)$, with $T>0$. Let $(a_i)_{i=1\dots N}$ satisfy \eqref{eq:no null partial sum}. The definition of the barycenter of clusters with intensities $a_i$ is given analogously to~\eqref{def:B_P}.

We assume that these points evolve such that there exists $C_0, C_1, C_2 \ge 0$ and $\alpha\geq 0$ such that
\begin{equation}\label{hyp:slow center of vorticity_v2}
    \forall P \in \cP_0(N),\qquad \left| \der{}{t} B_P(t)\right| \leq \sum_{i\in P} \sum_{j\notin P} \frac{C_0}{|\zeta_i(t)-\zeta_j(t)|^\alpha} + \sum_{i\in P}\frac{C_1}{|\zeta_i(t)|^\alpha}+C_2.
\end{equation}
Then there exists a constant $C_3>0$ such that for all $\eta \in (0,1]$, for all $t \in [0,T)$ such that
\begin{equation}
    T-t \le C_3\, \eta^{\alpha+1},
\end{equation}
and for all indices $i\neq j  \in \{1,\dots, N\}$,  the two following implications are true:
\begin{equation}
    |\zeta_i(t)-\zeta_j(t)| \geq \eta \quad\Longrightarrow\quad \forall \tau \in [t,T), \quad |\zeta_i(\tau)-\zeta_j(\tau)| \ge \frac{\eta}{2}
\end{equation}
and
\begin{equation}
    |\zeta_i(t)| \geq \eta \quad\Longrightarrow\quad \forall \tau \in [t,T), \quad |\zeta_i(\tau)| \ge \frac{\eta}{2}.
\end{equation}
The constant $C_3$ depends only on $\alpha$, $a$, $A$, $C_0$, $C_1$, $C_2$ and $N$.
\end{proposition}
The proof of this proposition is somehow similar to the proof of Proposition~\ref{prop:prevent collapse}. 
Several extra arguments are required to take into account the additional singular term appearing at~\eqref{hyp:slow center of vorticity_v2}. 
The details of the new proof are delayed to Section~\ref{sec:appendix holder}.

We continue with the proof of Theorem~\ref{thrm:bounded}.
\begin{lemma}\label{lem:traj are holder v2}
Let $t\in[0,T)\mapsto \big(x_i(t)\big)_i$ be a solution of the point-vortex problem in a bounded domain $\Omega$ with intensities $a_i$ satisfying the non-neutral clusters hypothesis~\eqref{eq:no null partial sum}.
Recall the definition of $I$ at \eqref{def:I again}, $\delta>0$ at~\eqref{def:delta again} and $T_\delta<T$ at~\eqref{def:T delta}. Define
\begin{equation}
    z_i(t):=\dist\big(x_i(t);\partial\Omega\big).
\end{equation}

Then, there exists a constant $C$ such that for all $i \in I$,
\begin{equation}\label{eq:holder avec le bord et T delta}
    \forall\;t_1<t_2\in[T_\delta, T),\qquad\big|z_i(t_2)-z_i(t_1)\big|\le \;C\;\sqrt{t_2-t_1\,}.
\end{equation}
\end{lemma}

\begin{proof}The proof of this lemma is reminiscent of the proof of Lemma~\ref{lem:traj are holder}. 
Let $t_2\in[T_\delta, T).$ For all $t \in [T_\delta,t_2)$, we define the evolution system $\zeta_k(t)$ for $k\in I\cup(N+I)$ by
\begin{equation}
\zeta_{i}(t):=z_i(t),\qquad\text{and}\qquad\zeta_{N+i}(t):=z_{i}(t_2),\qquad\text{where }i\in I.\end{equation}
We associate to this system some intensities $b_k$ for $k\in I\cup(N+I)$ such that $b_i=a_i$ for $i\in I$ and such that the non-neutral clusters hypothesis~\eqref{eq:no null partial sum} holds for the full family $b_k$. 
Such a choice of $b_k$ is always possible since the non-neutral clusters hypothesis holds for the $a_i$. 
Now, let $P\subseteq I\cup(N+I)$ non empty and denote by $B_P$ the center of vorticity associated to this dynamics. Note that
\begin{equation}\begin{split}
        \der{}{t}B_P&=\bigg(\sum_{k\in P}b_k\bigg)^{-1}\sum_{k\in P}b_k\,\der{}{t}\zeta_k(t)=\bigg(\sum_{k\in P}b_k\bigg)^{-1}\sum_{i\in P\cap I}a_i\,\der{}{t}z_i(t)=\bigg(\sum_{k\in P}b_k\bigg)^{-1}\bigg(\sum_{i\in P\cap I}a_i\bigg)\der{}{t}B_{P\cap I}.
    \end{split}
\end{equation}
We now observe that it is a consequence of Lemma~\ref{lem:centers of vorticity for the distances}, that the dynamics $t\mapsto z_i(t)$ for $i\in I$ satisfies the bound~\eqref{hyp:slow center of vorticity_v2} with $p=1$ and $\alpha=1$. This eventually implies
\begin{equation}\begin{split}
    \bigg|\der{}{t}B_P\bigg|&\leq\sum_{i\in P\cap I}\sum_{j\in I\setminus P}\frac{C_0'}{|z_i(t)-z_j(t)|}+\sum_{i\in P\cap I}\frac{C_1'}{|z_i(t)|}+C_2'\\
    &\leq\sum_{k\in P}\sum_{\ell\in [I\cup(N+I)]\setminus P}\frac{C_0'}{|\zeta_k(t)-\zeta_\ell(t)|}+\sum_{k\in P}\frac{C_1'}{|\zeta_k(t)|}+C_2'.
\end{split}
\end{equation}
Let $i\in I$.
We apply Proposition \ref{prop:prevent collapse_v2} with $\alpha = 1$ and $\eta:=|\zeta_i(t)-\zeta_{N+i}(t)|$ to the dynamics of the $\zeta_k$ on the interval of time $[T_\delta, t_2)$. Observe that we always have $\eta\leq 1$ on the interval $[T_\delta, T)$ as a consequence of $\delta\leq 1/4$.

Since by continuity of the trajectories we have that
\begin{equation}|\zeta_i(t)-\zeta_{N+i}(t)|=|z_i(t)-z_i(t_2)|\,\longrightarrow\,0,\qquad\text{as }t\to t_2^-,\end{equation}
in view of Proposition \ref{prop:prevent collapse_v2}, this gives that $t_2-t > C_3\, |\zeta_i(t)-\zeta_{N+i}(t)|^{2}$. Indeed, otherwise it would imply that $|\zeta_i(\tau)-\zeta_{N+i}(\tau)| \ge \eta/2$
for all $\tau \in [t,t_2]$. Therefore we have that
\begin{equation}|\zeta_i(t)-\zeta_{N+i}(t)|=|z_i(t)-z_i(t_2)|\leq C\,\sqrt{t_2-t\,}.\end{equation}

The estimate above then gives~\eqref{eq:holder avec le bord et T delta} and concludes the proof.
\end{proof}

To conclude the proof of Theorem~\ref{thrm:bounded}-$(ii)$, we combine the Hölder estimate provided by Lemma~\ref{lem:traj are holder v2} with the remark that on the time interval $[0,T_\delta]$ the trajectories are $C^\infty$.  One can directly check that the dependency of the Hölder constant with respect to the initial datum actually reduces to a dependency with respect to $d_0$ defined at \eqref{def:d_0}. 

To conclude the proof of Theorem~\ref{thrm:bounded}, there only remains to establish Proposition~\ref{prop:prevent collapse_v2}.

%%%%%%%%%%%%%%%%%%%%%%%%%%%%%%%%%%%%%%%%%%%%%%%%%%%%%%%%%%%%%%%%%%%

\subsection{Proof of Proposition~\ref{prop:prevent collapse_v2}}~\label{sec:appendix holder}

%%%%%%%%%%%%%%%%%%%%%%%%%%%%%%%%%%%%%%%%%%%%%%%%%%%%%%%%%%%%%%%%%%%
%%%%%%%%%%%%%%%%%%%%%%%%%%%%%%%%%%%%%%%%%%%%%%%%%%%%%%%%%%%%%%%%%%%

We first prove a preliminary Lemma.

\begin{lemma}[Cluster close to zero]\label{lem:cluster_zero}
Let $\mP$ be any partition of $\{1,\ldots,N\}$ satisfying the existence of $\kappa \in (0,\frac{1}{4})$, $\delta > 0$ and points $\zeta_i \in \RR^p$ such that
        \begin{equation}\label{eq:lemme_1}
        \forall\; P \in \mP, \quad \forall\;k,\ell \in P, \qquad |\zeta_k-\zeta_\ell| \leq \delta
        \end{equation}     
and
        \begin{equation}\label{eq:lemme_2}
        \forall P \neq P' \in \mP, \quad \forall k \in P, \quad \forall \ell \in P', \quad |\zeta_k-\zeta_\ell| \geq \kappa^{-1} \delta.
        \end{equation}
Then, there exists a unique set $\widehat{P} \in \mP \cup \{\emptyset\}$ satisfying the following properties.        
\begin{itemize}
    \item[(i)] The following relation holds :
    \begin{equation}
    \forall k \notin \widehat{P}, \quad |\zeta_k| \ge \kappa^{-1}\delta/4
\end{equation}
    \item[(ii)] If $\widehat{P}\neq \emptyset$, the following relation holds:
    \begin{equation}
        \exists \ell \in \widehat{P}, \quad |\zeta_\ell| < \kappa^{-1}\delta/4.
    \end{equation}
    \item[(iii)] The following relation holds:
    \begin{equation}
        \forall j \in \widehat{P}, \quad |\zeta_j| < (\kappa^{-1}/4 +1)\delta.
    \end{equation}
\end{itemize}
\end{lemma}

\begin{proof}
If there exists $\ell \in \{1,\ldots,N\}$ such that $|\zeta_\ell| < \kappa^{-1}\delta /4$, then we denote by $\widehat{P}$ the cluster containing $\ell$ in the partition $\mP$, else we set $\widehat{P} = \emptyset$. If $\widehat{P} = \emptyset$, that means that $(i)$ must be satisfied. Conditions $(ii)$ and $(iii)$ are true since they are empty. If $\widehat{P} \neq \emptyset$, by construction, $(ii)$ is satisfied and gives $(iii)$ by \eqref{eq:lemme_1}. Using relation \eqref{eq:lemme_2}, we have for every $k \notin\widehat{P}$,
\begin{equation}
    |\zeta_k - \zeta_\ell| \ge \kappa^{-1}\delta.
\end{equation}
Therefore
\begin{equation}
    |\zeta_k| = |\zeta_k-\zeta_\ell| - |\zeta_\ell| \ge \kappa^{-1}\delta - \kappa^{-1}\delta /4 \ge \kappa^{-1}\delta /4.
\end{equation}
Thus $(i)$ is satisfied. It is clear that $(i)$ and $(ii)$ define a unique set $\widehat{P} \in \mP \cup \{ \emptyset \}$.
\end{proof}

%%%%%%%%%%%%%%%%%%%%%%%%%%%%%%%%%%%%%%%%%%%%%%%%%%%%%%%%%%%%%%%%%%%
%%%%%%%%%%%%%%%%%%%%%%%%%%%%%%%%%%%%%%%%%%%%%%%%%%%%%%%%%%%%%%%%%%%
We continue with the proof of Proposition~\ref{prop:prevent collapse_v2}.
We proceed as in the proof of Proposition \ref{prop:prevent collapse}. 
 We can assume without loss of generality that $\max\{C_0, C_1, C_2 \}>0$. 
We fix once and for all some $\eta\in(0,1]$. Let $t_1$ such that
\begin{equation}
    T-t_1 \le C_3\, \eta^{\alpha+1}.
\end{equation} 
for some constant $C_3$ to be chosen later. During the proof, we will impose several conditions on the constant $C_3$ and at the end of the proof we will observe that all these conditions can be satisfied for a constant $C_3$ which is independent of $\eta$.

Let $0<\kappa < \frac{A}{16a}$. Recall that $A$ and $a$ are respectively defined by relations \eqref{def:A} and \eqref{def:a} and that $A \neq 0$ by hypothesis~\eqref{eq:no null partial sum}. Remark that $\kappa\leq 1/16$ since $A\leq a$. 

We start by constructing partitions of $\{1,\dots, N\}$ with an iterative process. 
We first invoke the corollary of the balls lemma (Corollary~\ref{coro:balls}) to the points $\zeta_k(t_1)$ to build the first partition $\mP^1$ by choosing $d := \frac{1}{8}\kappa \eta$.
This gives the first partition $\mP^1$ and a real number $\delta_1$ satisfying
\begin{equation}\label{eq:delta_1}
    \frac{1}{8}\left(\frac{\kappa}{8}\right)^N\kappa\eta\leq \delta_1 < \frac{1}{8}\kappa\eta
\end{equation}
such that
\begin{equation}
\forall\; P \in \mP^1,\quad \forall\; k,\ell \in P, \qquad |\zeta_k(t_1)-\zeta_\ell(t_1)| \leq \delta_1
\end{equation}
and
\begin{equation}
\forall P \neq P' \in \mP^1,\quad \forall k \in P, \quad \forall \ell \in P', \quad |\zeta_k(t_1)-\zeta_\ell(t_1)| \ge \kappa^{-1}\delta_1.
\end{equation}
We now define
\begin{equation}
    r := \min \left\{ \frac{1}{8}\,;\,\frac{A}{8a\kappa}-2\right\}>0,
\end{equation}
and
\begin{equation}
    s := r\left(\frac{\kappa}{8}\right)^N.
\end{equation}
Note that $r\in(0,\frac{1}{8})$.

We now build iteratively a decreasing sequence of positive numbers $\delta_q$, an increasing finite sequence of times $(t_q)$ and a finite number of partitions $\mP^q$ satisfying \vspace{0.1cm}
        \begin{equation}\label{(ii_v2)} \tag{$I$}
        \forall\; P \in \mP^q, \quad \forall\;k,\ell \in P, \qquad |\zeta_k(t_q)-\zeta_\ell(t_q)| \leq \delta_q
        \end{equation}     
and
        \begin{equation}\label{(iii_v2)} \tag{$II$}
        \forall P \neq P' \in \mP^q, \quad \forall k \in P, \quad \forall \ell \in P', \quad |\zeta_k(t_q)-\zeta_\ell(t_q)| \geq \kappa^{-1} \delta_q.
        \end{equation}
        
Assuming that $\mP^q$ is constructed with those properties we can define $\widehat{P}_q \in \mP^q \cup \{\emptyset\}$ by applying Lemma~\ref{lem:cluster_zero} to $\mP^q$. This means that $\widehat{P}_q$ satisfies
    \begin{equation}\begin{cases}\label{eq:prop_p_chapeau}
    \forall k \notin \widehat{P}_q, \quad |\zeta_k(t_q)| \ge \kappa^{-1}\delta_q/4\\
    \forall j \in \widehat{P}_q, \quad |\zeta_j(t_q)| < (\kappa^{-1}/4 +1)\delta_q.\end{cases}
    \end{equation}
The construction proceeds as follows. If the following relations are satisfied
        \begin{equation}\label{(vi_v2)}\tag{$\star$}
        \begin{cases}
           \displaystyle \forall\;k \notin \widehat{P}_q , \quad  \forall \tau \in [t_q,T), \qquad |\zeta_k(\tau)-\zeta_k(t_q)| \leq \kappa^{-1}\delta_q/8 \vspace{1mm}\\

            \displaystyle \forall k \in  \widehat{P}_q, \quad \forall \tau \in [t_q,T), \qquad |\zeta_k(\tau)| \leq \frac{3}{8} \kappa^{-1}\delta_q,
        \end{cases}
        \end{equation}
then the construction stops. Else, we will construct a time $t_{q+1}$ satisfying
        \begin{equation}\label{(v_v2)}\tag{$III$}
        \begin{cases}
           \displaystyle \forall\;k \notin  \widehat{P}_q,\quad \forall \tau \in [t_q,t_{q+1}], \qquad |\zeta_k(\tau)-\zeta_k(t_q)| \leq \kappa^{-1}\delta_q/8 \vspace{1mm}\\

            \displaystyle \forall\; k \in  \widehat{P}_q, \quad \forall \tau \in [t_q,t_{q+1}], \qquad |\zeta_k(\tau)| \leq \frac{3}{8} \kappa^{-1}\delta_q,
        \end{cases}
        \end{equation}
a real number $\delta_{q+1}$ satisfying
        \begin{equation}\label{(iv_v2)}\tag{$IV$}
              s \delta_q \le \delta_{q+1} < r \delta_q
        \end{equation}
and the next partition $\mP^{q+1}$ satisfying \eqref{(ii_v2)} and \eqref{(iii_v2)} at step $q+1$ as well as

\begin{equation}\label{(i_v2)}\tag{$V$}
    \mP^{q+1} \text{ is a sub-partition of } \mP^{q+1}
\end{equation}
and
\begin{equation}\label{(vii_v2)}\tag{$VI$}
   \mP^{q+1} \neq \mP^{q} \quad  \text{ or }  \quad \big( \widehat{P}_q \neq \emptyset \text{ and }\widehat{P}_{q+1} = \emptyset\big).
\end{equation}\vspace{0.1cm}

Let us observe that condition \eqref{(vii_v2)} ensures that the construction has a finite number of steps since we will prove that necessarily, $\widehat{P}_{q+1} \subset \widehat{P}_q$ and thus if this set is empty at one step, it will remain empty afterwards. Therefore $\mP^{q+1}$ is a strict sub-partition of $\mP^q$ except for at most one step. 

Let $q\in\NN^\ast$ be fixed and assume that the partitions $\mP^{q'}$ are constructed for all $q'=1\dots q$. Assume that \eqref{(vi_v2)} is not satisfied. We now construct $t_{q+1}$, $\delta_{q+1}$ and $\mP^{q+1}$. 

Since \eqref{(vi_v2)} is not satisfied, we can define the time $t_{q+1} \in [t_q,T)$ as being the largest time such that
\begin{equation}
\forall\;l\notin \widehat{P}_q, \quad \forall\;\tau\in[t_q,t_{q+1}],\qquad |\zeta_\ell(\tau) - \zeta_\ell(t_{q})| \leq \kappa^{-1}\delta_q/8
\end{equation}
and 
\begin{equation}
\forall\;k \in  \widehat{P}_q, \quad \forall\;\tau \in [t_q,t_{q+1}], \qquad |\zeta_k(\tau)| \leq \frac{3}{8} \kappa^{-1}\delta_q.
\end{equation}
By continuity of the trajectories, such a time $t_{q+1}$ does exist. This definition ensures that~\eqref{(v_v2)} holds true.

To define the new partition $\mP^{q+1}$, we apply Corollary \ref{coro:balls} to the points $(\zeta_j(t_{q+1}))_j$ with $d = r \delta_q$. This gives the partition $\mP^{q+1}$ and a real number $\delta_{q+1}>0$ such that~\eqref{(ii_v2)}, \eqref{(iii_v2)} hold true at step $q+1$ and~\eqref{(iv_v2)} holds true. 

We now prove \eqref{(i_v2)}.
For any $m \in P, n \in P'$, with $P\neq P' \in \mP^q \setminus \{\widehat{P}_q\}$ we have that
\begin{align*}
    |\zeta_m(t_{q+1}) - \zeta_n(t_{q+1})| & = |\zeta_m(t_q) - \zeta_n(t_q) + \zeta_m(t_{q+1}) - \zeta_m(t_q) + \zeta_n(t_q) - \zeta_n(t_{q+1})| \\
    & \ge |\zeta_m(t_q) - \zeta_n(t_q)| - | \zeta_m(t_{q+1}) - \zeta_m(t_q)| - |\zeta_n(t_q) - \zeta_n(t_{q+1})|.
    \end{align*}
We bound the first term using Hypothesis~\eqref{(iii_v2)} and the other two terms by using Hypothesis~\eqref{(v_v2)} to obtain that
    \begin{equation}
    |\zeta_m(t_{q+1}) - \zeta_n(t_{q+1})|  \ge \kappa^{-1} \delta_q - 2 \kappa^{-1}\delta_q/8 \ge \kappa^{-1} \delta_q/2.
\end{equation}
By Hypothesis~\eqref{(iv_v2)}, we know that $\delta_{q+1} < r \delta_q$. Since $r < 1/8$ and $ 16 \le \kappa^{-1}$, we infer that
\begin{equation}\label{eq:rapport_deltas}
    \kappa^{-1} \delta_q > 128\, \delta_{q+1} \quad \text{ and } \quad \kappa^{-1}\delta_q > 16 (\kappa^{-1}/4 + 1) \delta_{q+1}  
\end{equation}
Consequently, we have that
\begin{equation}
     |\zeta_m(t_{q+1}) - \zeta_n(t_{q+1})|  > \delta_{q+1}.
\end{equation}
In light of Hypothesis~\eqref{(ii_v2)} at step $q+1$, this proves that $m$ and $n$ do not belong to the same cluster in the partition $\mP^{q+1}$.

Now if $m \in P \in \mP^q \setminus \{\widehat{P}_q\}$ and $n \in \widehat{P}_q$, then 
\begin{align*}
    |\zeta_m(t_{q+1}) - \zeta_n(t_{q+1})| & = |\zeta_m(t_{q+1}) - \zeta_m(t_q) + \zeta_m(t_q) - \zeta_n(t_q) + \zeta_n(t_q) - \zeta_n(t_{q+1})| \\
    & \ge |\zeta_m(t_q) - \zeta_n(t_q)| - |\zeta_m(t_{q+1}) - \zeta_m(t_q)| - |\zeta_n(t_{q})| - |\zeta_n(t_{q+1})|.
\end{align*}
We bound the first term using Hypothesis~\eqref{(iii_v2)} and the other three using Hypothesis~\eqref{(v_v2)} to obtain that
\begin{equation}
    |\zeta_m(t_{q+1}) - \zeta_n(t_{q+1})|
    \ge \kappa^{-1}\delta_q - \kappa^{-1}\delta_q/8 - \frac{3}{8}\kappa^{-1}\delta_q - \frac{3}{8}\kappa^{-1}\delta_q = \kappa^{-1}\delta_q/8.
\end{equation}
Recalling relation~\eqref{eq:rapport_deltas}, this gives that
\begin{equation}
    |\zeta_m(t_{q+1}) - \zeta_n(t_{q+1})|  > \delta_{q+1}.
\end{equation}
In conclusion, if $m$ and $n$ do not belong to the same cluster in $\mP^q$, they do not belong to the same cluster in $\mP^{q+1}$. This proves that $\mP^{q+1}$ is a sub-partition of $\mP^q$. Hence \eqref{(i_v2)} is proved. \vspace{2mm}

There only remains to prove \eqref{(vii_v2)}. We start by proving an inclusion property.
Since $\mP^{q+1}$ satisfies~\eqref{(ii_v2)} and~\eqref{(iii_v2)}, we can already define $\widehat{P}_{q+1}$. Let us prove that
\begin{equation}\label{eq:inclusion_des_p_chapeau}
    \widehat{P}_{q+1} \subset \widehat{P}_q.
\end{equation} 
Indeed, $\forall k \notin \widehat{P}_q$, by \eqref{eq:prop_p_chapeau} and Hypothesis \eqref{(v_v2)} we have that 
\begin{equation}
    |\zeta_k(t_{q+1})| \ge |\zeta_k(t_q)| - |\zeta_k(t_q) - \zeta_k(t_{q+1})| \ge \kappa^{-1}\delta_q (1/4 - 1/8) = \kappa^{-1}\delta_q/8.
\end{equation}
Recalling once again relation \eqref{eq:rapport_deltas}, we have that
\begin{equation}
    |\zeta_k(t_{q+1})| > (\kappa^{-1}/4+1)\delta_{q+1}.
\end{equation}
Relations \eqref{eq:prop_p_chapeau} at step $q+1$ conclude that $k \notin \widehat{P}_{q+1}$. Therefore $\widehat{P}_{q+1} \subset \widehat{P}_q$.\vspace{2mm}

It is not necessarily true that $\mP^{q+1} \neq \mP^q$ for all $q$. In order to prove \eqref{(vii_v2)}, we need to separate the analysis into two cases. Since $t_{q+1}$ is the largest time such that Hypothesis~\eqref{(v_v2)} hold true, we either have that 
\begin{equation}\label{hyp:case2v2-1}
    \exists\, k \in \widehat{P}_q, \qquad |\zeta_k(t_{q+1})| = \frac{3}{8}\kappa^{-1}\delta_{q},
\end{equation}
we denote this Case 1, or that
\begin{equation}\label{majtq+1_v2}
  \exists\, k \notin\widehat{P}_q, \qquad |\zeta_k(t_{q+1})) - \zeta_k(t_{q})| =\kappa^{-1}\delta_q/8.
\end{equation}
This is Case 2.\vspace{2mm}

\vspace{3mm}

%%%%%%%%%%%%%%%

%%%%%%%%%%%%%%%%%%%%

$\bullet\;$\textbf{Case 1:}
Assume that \eqref{hyp:case2v2-1} holds true for a given $k \in \widehat{P}_q$. In particular, this requires that $\widehat{P}_q \neq \emptyset$.

If $\mP^{q+1} \neq \mP^q$, then \eqref{(vii_v2)} is proved. 
Recall that~\eqref{(ii_v2)},~\eqref{(iii_v2)},~\eqref{(v_v2)},~\eqref{(iv_v2)} and~\eqref{(i_v2)} are proved. We have thus correctly constructed $\mP^{q+1}$ and end Case 1 here. \vspace{1mm}

Assume now that $\mP^{q+1} = \mP^q$. In this case we need to prove that $\widehat{P}_{q+1} = \emptyset$. 

Let $\ell \in \widehat{P}_q$.  Since $\widehat{P}_q \in \mP^q$, we have in particular that $\widehat{P}_q \in \mP^{q+1}$. Consequently, by condition~\eqref{(ii_v2)} used at step $q+1$,
\begin{equation}
    | \zeta_k(t_{q+1}) - \zeta_\ell(t_{q+1})| \le \delta_{q+1}.
\end{equation}
Combining this with the fact that $k$ satisfies \eqref{hyp:case2v2-1} gives that
\begin{align*}
    |\zeta_\ell(t_{q+1})| & = |\zeta_\ell(t_{q+1}) - \zeta_k(t_{q+1}) + \zeta_k(t_{q+1})| \\
    & \ge |\zeta_k(t_{q+1})| - |\zeta_\ell(t_{q+1}) - \zeta_k(t_{q+1})|\\
    & \ge \frac{3}{8}\kappa^{-1}\delta_q - \delta_{q+1}.
\end{align*}
Recalling relation \eqref{eq:rapport_deltas}, we observe that
\begin{equation}
    \frac{3}{8}\kappa^{-1}\delta_q - \delta_{q+1} \ge (\kappa^{-1}/4+1)\delta_{q+1}.
\end{equation}
Therefore,
\begin{equation}
    |\zeta_\ell(t_{q+1})| \ge (\kappa^{-1}/4+1)\delta_{q+1}.
\end{equation}
This relation combined with relation \eqref{eq:prop_p_chapeau} at step $q+1$ implies that $\ell \notin \widehat{P}_{q+1}$. Therefore any $l \in \widehat{P}_q$ satisfies $\ell \notin \widehat{P}_{q+1}$. By relation~\eqref{eq:inclusion_des_p_chapeau} this means that $\widehat{P}_{q+1}= \emptyset$. Therefore, \eqref{(vii_v2)} is proved. This concludes Case~1.

%%%%%%%%%%%%%%%%%%

$\bullet\;$\textbf{Case 2:} 
Assume that \eqref{majtq+1_v2} holds true for a given $k \notin \widehat{P}_q$. Let us prove that $\mP^{q+1} \neq \mP^q$.

Let $P \in \mP^{q}$ such that $k \in P$. For $\tau \in [t_q,t_{q+1}]$ and $j \in \widehat{P}_q$, by Hypothesis~\eqref{(v_v2)} we have that
\begin{equation}\label{tagada}
    |\zeta_j(\tau)-\zeta_j(t_q)|  \le |\zeta_j(\tau)| + |\zeta_j(t_q)| \leq\frac{6}{8}\kappa^{-1} \delta_q.
\end{equation}
This last relation is also true for $j \notin \widehat{P}_q$ because of the stronger relation \eqref{(v_v2)}. Then, for any $i\in P$, any $j \notin P$ and any $\tau \in [t_q,t_{q+1}]$ we have that
\begin{equation}
    |\zeta_i(\tau) - \zeta_j(\tau)|  \ge |\zeta_i(t_q) - \zeta_j(t_q)| - |\zeta_i(\tau) - \zeta_i(t_q)| - |\zeta_j(\tau) - \zeta_j(t_q)|.
\end{equation}
We bound the first term using Hypothesis~\eqref{(iii_v2)}, then recalling that $P \neq \widehat{P}_q$, we bound the second term using Hypothesis~\eqref{(v_v2)} and the third with relation~\eqref{tagada} to obtain that
\begin{equation}\label{toto1}
     |\zeta_i(\tau) - \zeta_j(\tau)|   \ge \kappa^{-1}\delta_q - \frac{1}{8}\kappa^{-1}\delta_q - \frac{6}{8}\kappa^{-1} \delta_q  =  \kappa^{-1}\delta_q /8.
\end{equation}
Relation \eqref{eq:prop_p_chapeau} and Hypothesis~\eqref{(v_v2)} give that for every $\tau \in [t_q,t_{q+1}]$ and for every $i \in P$,
\begin{equation}\label{toto2}
    |\zeta_i(\tau)| \ge |\zeta_i(t_q)| - |\zeta_i(t_q)-\zeta_i(\tau) | \ge \kappa^{-1}\delta_q /4-\kappa^{-1}\delta_q /8 = \kappa^{-1}\delta_q /8.
\end{equation}
Recalling relation \eqref{hyp:slow center of vorticity_v2} we have that
\begin{equation}
     \left| \der{}{t} B_P(\tau)\right| \leq \sum_{i\in P} \sum_{j\notin P} \frac{C_0}{|\zeta_i(\tau)-\zeta_j(\tau)|^\alpha} + \sum_{i\in P}\frac{C_1}{|\zeta_i(\tau)|^\alpha}+C_2.
\end{equation}
Plugging relations \eqref{toto1} and \eqref{toto2} into this last relation gives
\begin{align*}
    \left| \der{}{t} B_P(\tau)\right| & \leq \sum_{i\in P} \sum_{j\notin P} \frac{C_0}{(\kappa^{-1}\delta_q/8)^\alpha} + \sum_{i\in P}\frac{C_1}{(\kappa^{-1}\delta_q/8)^\alpha}+C_2 \\
    & \leq \frac{8^\alpha N^2 (C_0+C_1)}{(\kappa^{-1}\delta_q)^\alpha} + C_2.
\end{align*}

Up to a multiplicative constant, this estimate is identical to the estimate~\eqref{de Cadix a des yeux de velours} obtained in the proof of Proposition \ref{prop:prevent collapse}. Provided that the constant $C_3$ is small enough, the same argument as in page \pageref{eq:points s'eloignent} leads to $\mP^{q+1} \neq \mP^q$. 
The hypothesis that is required on $C_3$ in this case, which plays the role of relation \eqref{j ai soif}, is:
\begin{equation}\label{hyp:C3}
    C_3 \eta^{\alpha+1} \le \frac{a}{A}\delta_q \left( \frac{8^\alpha \, N^2 \, (C_0+C_1)}{(\kappa^{-1}\delta_q)^\alpha} + C_2\right)^{-1}.
\end{equation}
Condition \eqref{(vii_v2)} is proved. This concludes Case 2 and our construction. 

It is clear that the situation 
\begin{equation}
\widehat{P}_q \neq \emptyset \text{ and } \widehat{P}_{q+1} = \emptyset
\end{equation}
can happen at most one time since we have proved relation \eqref{eq:inclusion_des_p_chapeau} holds at every step of the construction. Therefore this iterative process has at most $N$ steps, namely $q = 1\ldots Q$ with $Q \le N+1$. We recall that Hypothesis~\eqref{(vi_v2)} must hold true at the final step $Q$. \vspace{3mm}

Recalling that $\delta_1$ satisfies~\eqref{eq:delta_1}, we can prove by induction, using the construction condition \eqref{(iv_v2)} and the fact that $s < \left(\frac{\kappa}{8}\right)^N$, that
        \begin{equation}\label{eq:encadrement_delta_v2}
              \frac{1}{8}s^{q}\kappa \eta \leq\delta_q \leq  r^{q-1} \kappa \eta \frac{1}{8}.
        \end{equation}
        
We now conclude the proof of Proposition~\ref{prop:prevent collapse_v2}. We first establish the following intermediate property: for every $i \in \{1,\ldots,N\}$ and for every $\tau \in [t_1,T)$,
\begin{equation}\label{Pas bouger}
    |\zeta_i(\tau) - \zeta_i(t_1)| \le \eta/4.
\end{equation}

Indeed, in the case where $i \notin \widehat{P}_1$, we first observe that $\forall q \in \{1,\ldots,Q\}$, $i \notin \widehat{P}_q$. If we set the convention that $t_{Q+1} = T$, then there exists a unique $q \in \{1,\ldots,Q\}$ such that $\tau \in [t_q,t_{q+1})$. We write
\begin{equation}
    |\zeta_i(\tau) - \zeta_i(t_1)| \leq  \sum_{q'=1}^{q-1} |\zeta_i(t_{q'+1})-\zeta_i(t_{q'} )| + |\zeta_i(\tau) - \zeta_i(t_q)|.
\end{equation}
Then, by construction hypothesis~\eqref{(v_v2)} and~\eqref{(vi_v2)} we have that
\begin{equation}
    |\zeta_i(\tau) - \zeta_i(t_1)| \leq \sum_{q'=1}^{q}\kappa^{-1}\delta_{q'}/8.
\end{equation}
By relation \eqref{eq:encadrement_delta_v2} and the fact that $r<1/8$ this yields
\begin{equation}
    |\zeta_i(\tau) - \zeta_i(t_1)| \leq \frac{1}{8}\sum_{q'=1}^q r^{q'-1} \eta/8 \leq \eta/4.
\end{equation}

If on the contrary $i \in \widehat{P}_1$, then we define $q \in \{1,\ldots,Q\}$ as being the greatest index such that $i \in \widehat{P}_q$. 
We have by construction hypothesis~\eqref{(v_v2)} and~\eqref{(vi_v2)} and the fact that the sequence $\delta_q$ is decreasing that for any $\tau \in [t_1,t_{q+1})$
\begin{equation}
    |\zeta_i(\tau) - \zeta_i(t_1)| \le |\zeta_i(\tau)| + |\zeta_i(t_1)| \le 2 \frac{3}{8}\kappa^{-1}\delta_1 \le \eta/8.
\end{equation}
Moreover by Hypothesis~\eqref{(v_v2)}, \eqref{(vi_v2)} and relation~\eqref{eq:encadrement_delta_v2} we have that for $\tau \ge t_{q+1}$ that
\begin{equation}\label{eq:da_one}
    |\zeta_i(\tau) - \zeta_i(t_{q+1})| \le \sum_{q'=q+1}^\infty \kappa^{-1}\delta_{q'} /8 \le \eta/8
\end{equation}
so that
\begin{equation}\label{eq:da_two}
    |\zeta_i(\tau) - \zeta_i(t_1)| \le |\zeta_i(\tau) - \zeta_i(t_{q+1})| + |\zeta_i(t_{q+1}) - \zeta_i(t_1)| \le \eta/4.
\end{equation}
We have proved relation \eqref{Pas bouger}.

Let $i$ and $j$ such that $|\zeta_i(t_1) - \zeta_j(t_1)| \ge \eta$. Relation \eqref{Pas bouger} applied on $i$ and $j$ gives that 
\begin{align*}
    |\zeta_i(\tau)-\zeta_j(\tau)| & = |\zeta_i(\tau) - \zeta_i(t_1) + \zeta_i(t_1)-\zeta_j(t_1) + \zeta_j(t_1) - \zeta_j(\tau)| \\
    & \ge |\zeta_i(t_1)-\zeta_j(t_1)| - |\zeta_i(\tau) - \zeta_i(t_1)| - |\zeta_j(t_1) - \zeta_j(\tau)| \\
    & \ge \eta - 2\eta/4 =\eta/2.
\end{align*}
This proves relation~\eqref{eq:da_one} of Proposition~\ref{prop:prevent collapse_v2}.

Now if $i$ is such that $|\zeta_i(t_1)| \ge \eta$ then for every $\tau \in [t_1,T)$,
\begin{equation}
    |\zeta_i(\tau)| \ge |\zeta_i(t_1)| - |\zeta_i(t_1)-\zeta_i(\tau)| \ge \eta/2,
\end{equation}
as a consequence of relation \eqref{Pas bouger}. This proves~\eqref{eq:da_two}.

What remains to prove is the fact that $C_3$ does not depend on $\eta$. Proceeding as in the proof of Proposition~\ref{prop:prevent collapse}, starting page \eqref{eq:condition t reformulee}, we can take
\begin{equation}
    C_3 = \frac{a}{A} s^{(N+1)(\alpha+1)}\frac{\kappa}{8^\alpha\, N^2\,(C_0+C_1)+C_2}
\end{equation}
to ensure that relation \eqref{hyp:C3} hold for every $q \in \{1,\ldots,Q\}$. Choosing $\kappa$ and replacing $s$ by its value gives the announced constant $C_3$.\qed \vspace{1cm}

%%%%%%%%%%%%%%%%%%%%%%%%%%%%%%%%%%%%%%%%%%%%%%%%%%%%%%%%%%%%%%%%%%%
%%%%%%%%%%%%%%%%%%%%%%%%%%%%%%%%%%%%%%%%%%%%%%%%%%%%%%%%%%%%%%%%%%%

\textbf{\LARGE Appendix}

\appendix
%%%%%%%%%%%%%%%%%%%%%%%%%%%%%%%%%%%%%%%%%%%%%%%%%%%%%%%%%%%%%%%%%%%
%%%%%%%%%%%%%%%%%%%%%%%%%%%%%%%%%%%%%%%%%%%%%%%%%%%%%%%%%%%%%%%%%%%
%%%%%%%%%%%%%%%%%%%%%%%%%%%%%%%%%%%%%%%%%%%%%%%%%%%%%%%%%%%%%%%%%%%
\section{Optimality of the Hölder exponent and self-similar collapses}\label{appendix:optimality}
This appendix is devoted to the proof of the existence of collapses for all values of $\alpha>0$ and to check that the Hölder exponent given by Theorem~\ref{thrm:plane} is optimal. 
The first part of this appendix studies necessary conditions to have a self-similar collapse.
In the second part, we exhibit an example of self-similar collapse in the case of the $3$-vortex problem for any value of $\alpha>0$.

We say that a solution is a self-similar collapse at time $T>0$ if there exist $C^1$ maps $f:[0,T] \rightarrow \RR_+$ and $\theta : [0,T) \rightarrow \RR$ satisfying
\begin{equation}\label{hyp:self similar def f and theta}
    \begin{cases}
        f(0) = 1 \\
        f(T) = 0 \\
        \theta(0) = 0
    \end{cases}
\end{equation}
and for every $j \in \{ 1,\dots, N\}$, and every $t \in [0,T)$
\begin{equation}\label{hyp:self similar def x}
    x_j(t) = f(t)\, x_j(0)\, e^{\cu \theta(t)}
\end{equation}
where $\cu$ is the complex unit ($\cu^2 = -1$). Recall that taking the ``$\perp$'', the rotation of angle $\pi/2$ in $\RR^2$, is equivalent to the multiplication by $\cu$ in $\CC$. By continuity of $f$, the condition $f(T)=0$ implies that for every $j \in \{ 1,\dots, N\}$,
\begin{equation}\label{hyp:x(T)=0}
    x_j(t) \longrightarrow 0,\qquad\text{as }t\to T^-.
\end{equation}
Moreover, a necessary condition  for the collapse to happen exactly at time $T>0$ is to have $f(t) > 0$ for every $t <T$.

%%%%%%%%%%%%%%%%%%%%%%%%%%%%%%%%%%%%%%%%%%%%%%%%%%%%%%%%%%%%%%%%%%%%
%%%%%%%%%%%%%%%%%%%%%%%%%%%%%%%%%%%%%%%%%%%%%%%%%%%%%%%%%%%%%%%%%%%%
\subsection{Necessary conditions for a self-similar collapse}\label{section:appendix_necessary}
We consider the $\alpha$-point-vortex dynamic \eqref{eq:evo alpha}. The case $\alpha=1$ (corresponding to Euler point-vortices in the plane) being already well-known~\cite{Aref_1979, Aref_2010, Grotto_Pappalettera_2020, Hiraoka_2008, Krishnamurthy_Stremler_2018}, we fully concentrate here on the case $\alpha\neq 1$. 
The aim is to extract some necessary conditions on the intensities $a_i$ and on the starting positions $x_1(0), \ldots, x_N(0)$ of a configuration to have a self similar collapse.

We introduce $l_{ij} = |x_i-x_j|$. Since the Hamiltonian \eqref{def:Hamiltonien} is preserved during the motion, the first condition we obtain is that $H(t) = H(0)$ holds at all times. However relation \eqref{hyp:self similar def x} gives that
\begin{equation}
    H(t) = \frac{1}{f(t)^{\alpha-1}} H(0).
\end{equation}
Therefore the Hamiltonian must be equal to $0$, which reads
\begin{equation}\label{eq:CN Hamiltonian self similar collapse}
    \sum_{i\neq j} \frac{a_ia_j}{l_{ij}^{\alpha-1}(0)} = 0.
\end{equation}
The second invariant we have is
\begin{equation}
    L(t) = \sum_{i\neq j} a_ia_j l_{ij}^2(t).
\end{equation}
Indeed, $L = \sum_{i\neq j} a_i a_j |x_i-x_j|^2 = 2 \left(\sum_{i=1}^N a_i\right)I - 2|M|^2$ with $M$ defined by \eqref{def:M} and $I$ defined by \eqref{def:I(X)}. Both $M$ and $I$ are constant in time so $L$  is also constant in time. 
Relation \eqref{hyp:x(T)=0} gives that $l_{ij}(t)$ tends to $0$ as $t\to T^-$. 
This implies that $L(0)=0$, which reads
\begin{equation}\label{eq:CN L(X) self similar collapse}
    \sum_{i\neq j} a_ia_j l_{ij}^2(0)=0.
\end{equation}
Relations \eqref{eq:CN Hamiltonian self similar collapse} and \eqref{eq:CN L(X) self similar collapse} are two necessary conditions on the intensities and the starting positions for a self similar collapse to occur. These conditions are the generalization for all $\alpha$ of the already known conditions when $\alpha=1$ (see~\cite{Aref_2010}) and when $\alpha=2$ (see~\cite{Reinaud2021}).

We now go further to find the necessary expression of $f$ appearing in~\eqref{hyp:self similar def x}. In the case of self-similar collapses, using relation~\eqref{hyp:self similar def x}, we compute the evolution $l_{ij}^2$ for $\alpha$ models~\eqref{eq:evo alpha}:
\begin{align*}
    &\der{}{t}l_{ij}^2 = 2(x_i-x_j)\cdot \der{}{t} (x_i-x_j) \\
    & =2(x_i-x_j) \cdot\left( \sum_{k\neq i} a_k\frac{(x_i-x_k)^\perp}{|x_i-x_k|^{\alpha+1}} - \sum_{\ell \neq j}a_\ell\frac{(x_j-x_\ell)^\perp}{|x_j-x_\ell|^{\alpha+1}}  \right) \\
    & = 2\left[(x_i(0)-x_j(0))f(t)e^{\cu\theta}\right] \cdot \left[\frac{e^{\cu \theta}}{f(t)^{\alpha}}\left( \sum_{k\neq i} a_k\frac{(x_i(0)-x_k(0))^\perp}{|x_i(0)-x_k(0)|^{\alpha+1}} - \sum_{\ell \neq j}a_\ell\frac{(x_j(0)-x_\ell(0))^\perp}{|x_j(0)-x_\ell(0)|^{\alpha+1}}  \right)\right] \\
    &=\frac{2}{f(t)^{\alpha-1}}(x_i(0)-x_j(0)) \cdot \left( \sum_{k\neq i} a_k\frac{(x_i(0)-x_k(0))^\perp}{|x_i(0)-x_k(0)|^{\alpha+1}} - \sum_{\ell \neq j}a_\ell\frac{(x_j(0)-x_\ell(0))^\perp}{|x_j(0)-x_\ell(0)|^{\alpha+1}}  \right)
\end{align*}
Therefore, there exist constants $C_{i,j}$ independent of the time such that
\begin{equation}\label{eq:diff eq on f}
    f'(t)f(t) = \frac{C_{i,j}}{f(t)^{\alpha-1}} ,
\end{equation}
where
\begin{equation}
    C_{i,j} := \frac{x_i(0)-x_j(0)}{|x_i(0)-x_j(0)|^2} \cdot \left( \sum_{k\neq i} a_k\frac{(x_i(0)-x_k(0))^\perp}{|x_i(0)-x_k(0)|^{\alpha+1}} - \sum_{\ell \neq j}a_\ell\frac{(x_j(0)-x_\ell(0))^\perp}{|x_j(0)-x_\ell(0)|^{\alpha+1}}  \right).
\end{equation}
We now observe that the equality~\eqref{eq:diff eq on f} implies that the constants $C_{i,j}$ do not actually depend on $i$ and $j$. We denote $C = C_{i,j}$ and solve \eqref{eq:diff eq on f} to obtain
\begin{equation}
    f(T)^{\alpha+1}-f(t)^{\alpha+1} = (\alpha+1) C (T-t).
\end{equation}
Recalling from \eqref{hyp:self similar def f and theta} that $f(T) = 0$, we observe that we must have that $C<0$. This gives
\begin{equation}
    f(t) = C' (T-t)^\frac{1}{\alpha+1}
\end{equation}
with
\begin{equation}
    C' = (-C(\alpha+1))^\frac{1}{\alpha+1}.
\end{equation}
Recalling from \eqref{hyp:self similar def f and theta} that $f(0)=1$, we have $C' = \frac{1}{T^\frac{1}{\alpha+1}}$, so that we finally obtain that the function $f$ is necessarily equal to
\begin{equation}\label{eq:expression of f}
    f(t) = \left(\frac{T-t}{T}\right)^\frac{1}{\alpha+1}.
\end{equation}

We do the same for the function $\theta$. Taking the derivative in time of relation \eqref{hyp:self similar def x} gives
\begin{equation}
\sum_{k\neq j} \cu a_k\frac{x_j-x_k}{|x_j-x_k|^{\alpha+1}} = f'(t)x_j(0) e^{\cu\theta(t)} +  f(t)x_j(0)\cu\theta'(t)e^{\cu\theta(t)}.
\end{equation}
Thus
\begin{equation}
    \frac{1}{f(t)^\alpha}\sum_{k\neq j} a_k\frac{x_j(0)-x_k(0)}{|x_j(0)-x_k(0)|^{\alpha+1}} = -\cu f'(t)x_j(0) + f(t)x_j(0)\theta'(t).
\end{equation}
Recalling relations \eqref{eq:diff eq on f} and \eqref{eq:expression of f} we have that
\begin{equation}\label{eq:diff eq on theta}
 -\cu Cx_j(0) + \frac{T-t}{T}\theta'(t)x_j(0)=D_{j},
\end{equation}
where 
\begin{equation}
    D_{j} = \sum_{k\neq j} a_k \frac{x_j(0)-x_k(0)}{|x_j(0)-x_k(0)|^{\alpha+1}}.
\end{equation}
We remark that relation~\eqref{eq:diff eq on theta} implies that the quantity $\frac{T-t}{T}\theta'(t)$ is constant in time. We denote its value by $D\in\RR$. Recalling from \eqref{hyp:self similar def f and theta} that $\theta(0) = 0$ we infer that
\begin{equation}\label{eq:expression of theta}
    \theta(t) = -DT\ln\frac{T-t}{T}.
\end{equation}
Gathering relations \eqref{hyp:self similar def x}, \eqref{eq:expression of f} and \eqref{eq:expression of theta} we obtain that the trajectories of a self-similar collapse are necessary of the following form:
\begin{equation}
    x_j(t) = x_j(0) \left(\frac{T-t}{T}\right)^\frac{1}{\alpha +1}\exp\bigg({-\cu DT\ln\frac{T-t}{T}}\bigg).
\end{equation}
It is a direct computation to check that in the expression above, the function $x_j(t)$ belongs to the Hölder space $\cC^{0,\beta}([0,T];\CC)$ if and only if $\beta\leq\frac{1}{\alpha+1}$.

%%%%%%%%%%%%%%%%%%%%%%%%%%%%%%%%%%%%%%%%%%%%%%%%%%%%%%%%%%%%%%%%%%%
%%%%%%%%%%%%%%%%%%%%%%%%%%%%%%%%%%%%%%%%%%%%%%%%%%%%%%%%%%%%%%%%%%%
\subsection{Existence of a self-similar collapse}

We now prove that, for any $\alpha > 0$, there exists an initial configuration leading to a self-similar collapse for the point-vortex dynamic \eqref{eq:evo alpha}. 

We construct our example in the case $N=3$, $a_2=a_3=1$ and $a_1 = a \in \RR^\ast$ to be fixed later. We define $A = |x_2-x_3|$, $B = |x_3-x_1|$ and $C =| x_1 - x_2|$. 
Let $\lambda \in (0,1)$. 
We choose values for $x_1(0)$, $x_2(0)$ and $x_3(0)$ such that these three points form a direct orthogonal triangle with $A(0) = 1$, $B(0) = \lambda$ and $C(0) = \sqrt{\lambda^2+1}$. 
We now want to find a value for $\lambda$ and $a$ such that the associated solution is a self-similar collapse. More precisely, we are going to exhibit a value for $a$ and $\lambda$ such that $B(t) = \lambda A(t)$ and $C(t) = \sqrt{\lambda^2+1}A(t)$ hold for any time $t$, and such that $A(t)\to 0$ as $t\to T$ for some $T>0$.

It is a direct computation using the point-vortex equations~\eqref{eq:evo alpha} to check that (see also~\cite{Reinaud2021}), 
\begin{equation}
    \der{x_2}{t} - \der{x_3}{t}  = a\frac{(x_2-x_1)^\perp}{C^{\alpha+1}}+\frac{(x_2-x_3)^\perp}{A^{\alpha+1}} - a\frac{(x_3-x_1)^\perp}{B^{\alpha+1}}-\frac{(x_3-x_2)^\perp}{A^{\alpha+1}} 
\end{equation}
so that
\begin{equation}
    (x_2-x_3)\cdot \der{}{t} (x_2-x_3) = a \left(  \frac{(x_2-x_3)\cdot(x_2-x_1)^\perp}{C^{\alpha+1}} - \frac{(x_2-x_3)\cdot(x_3-x_1)^\perp}{B^{\alpha+1}} \right).
\end{equation}
In the case of the $3$ vortex problem~\cite{Aref_1979, Reinaud2021}, the equations above can be rewritten using $\bigtriangleup$, the area of the direct triangle $(x_1,x_2,x_3)$. We have that
\begin{equation}
    (x_2-x_3)\cdot(x_2-x_1)^\perp = (x_3-x_1)\cdot(x_3-x_2)^\perp = (x_1-x_2)\cdot(x_1-x_3)^\perp = -2\bigtriangleup.
\end{equation}
This gives:
\begin{equation}
    \der{A^2}{t} = 4a\!\bigtriangleup\! \left( \frac{1}{B^{\alpha+1}} - \frac{1}{C^{\alpha+1}} \right).
\end{equation}
Similar computations lead to
\begin{equation}
    \der{B^2}{t} = 4\!\bigtriangleup\!\left( \frac{1}{C^{\alpha+1}} - \frac{1}{A^{\alpha+1}} \right)
\end{equation}
and
\begin{equation}
    \der{C^2}{t} = 4\!\bigtriangleup\!\left( \frac{1}{A^{\alpha+1}} - \frac{1}{B^{\alpha+1}} \right).
\end{equation}
If we define $X := (A^2,B^2,C^2)$, then $t\mapsto X(t)$ is a solution of an autonomous differential equation:
\begin{equation}
    \der{X}{t} = F(X),
\end{equation}
where $F:\RR^3\to\RR^3$ is defined by the $3$ previous equations.

We now introduce another differential system with unknown $\tilde{X} = (\tilde{A}^2,\tilde{B}^2,\tilde{C}^2)$ defined by:
\begin{equation}
    \der{\tilde{X}}{t} = 4a\tilde{\bigtriangleup} \left( \frac{1}{\tilde{B}^{\alpha+1}} - \frac{1}{\tilde{C}^{\alpha+1}} \right)\begin{pmatrix} 1 \\  \lambda^2  \\  1+\lambda^2\end{pmatrix},
\end{equation}
where $\tilde{\bigtriangleup}$ is the area of a triangle which sides are $\tilde{A}$, $\tilde{B}$ and $\tilde{C}$.
We consider the same initial datum $\tilde{X}(0) = X(0)$. 
By construction we have that $\tilde{B} = \lambda \tilde{A}$ and $\tilde{C} = \sqrt{1+\lambda^2}\tilde{A}$ at every time. Indeed, these two relations hold at time 0 and hold for the the derivatives in time for all times.

The aim is to prove that $\tilde{X}$ satisfies the same differential system as $X$ for a well-chosen value of $\lambda$, so that we can deduce $X=\tilde{X}$. We need to prove that
\begin{equation}\label{eq:relation needed on B}
    \der{\tilde{B}^2}{t} = 4 \tilde{\bigtriangleup} \left( \frac{1}{\tilde{C}^{\alpha+1}} - \frac{1}{\tilde{A}^{\alpha+1}} \right),
\end{equation}
and
\begin{equation}\label{eq:relation needed on C}
    \der{\tilde{C}^2}{t} = 4 \tilde{\bigtriangleup} \left( \frac{1}{\tilde{A}^{\alpha+1}} - \frac{1}{\tilde{B}^{\alpha+1}} \right).
\end{equation}
Expressing $\der{\tilde{B}^2}{t}$, $\tilde{B}$ and $\tilde{C}$ in terms of $\tilde{A}$ in relation \eqref{eq:relation needed on B} gives that \eqref{eq:relation needed on B} is equivalent to
\begin{equation}
    4\lambda^2 a\tilde{\bigtriangleup} \left( \frac{1}{(\lambda\tilde{A})^{\alpha+1}} - \frac{1}{(\sqrt{1+\lambda^2}\tilde{A})^{\alpha+1}} \right)  = 4 \tilde{\bigtriangleup} \left( \frac{1}{(\sqrt{1+\lambda^2}\tilde{A})^{\alpha+1}} - \frac{1}{\tilde{A}^{\alpha+1}} \right).
\end{equation}
Simplifying, this is equivalent until the collapse to
\begin{equation}\label{eq:condition lambda 1}
    \lambda^2 a\left( \frac{1}{\lambda^{\alpha+1}} - \frac{1}{(1+\lambda^2)^\frac{\alpha+1}{2}} \right)  =  \frac{1}{(1+\lambda^2)^\frac{\alpha+1}{2}} - 1.
\end{equation}
Similarly, relation \eqref{eq:relation needed on C} is equivalent to
\begin{equation}\label{eq:condition lambda 2}
    a(1+\lambda^2) \left(  \frac{1}{\lambda^{\alpha+1}} - \frac{1}{(1+\lambda^2)^\frac{\alpha+1}{2}}\right)  =  1 - \frac{1}{\lambda^{\alpha+1}} .
\end{equation}
Relations \eqref{eq:condition lambda 1} and \eqref{eq:condition lambda 2} form a system of two equations on $a$ and $\lambda$:
\begin{equation}\label{eq:system a lambda}
    \begin{cases}
        a & = \frac{1}{\lambda^2}\left( \frac{1}{\lambda^{\alpha+1}} - \frac{1}{(1+\lambda^2)^\frac{\alpha+1}{2}} \right)^{-1}\left( \frac{1}{(1+\lambda^2)^\frac{\alpha+1}{2}} - 1\right) \\
        a & = \left( 1 - \frac{1}{\lambda^{\alpha+1}} \right) (1+\lambda^2)^{-1} \left(  \frac{1}{\lambda^{\alpha+1}} - \frac{1}{(1+\lambda^2)^\frac{\alpha+1}{2}}\right)^{-1}.
    \end{cases}
\end{equation}
Consequently, we want to find a solution $\lambda$ to the following equation:
\begin{multline*}
         \frac{1}{\lambda^2}\left( \frac{1}{\lambda^{\alpha+1}} - \frac{1}{(1+\lambda^2)^\frac{\alpha+1}{2}} \right)^{-1}\left( \frac{1}{(1+\lambda^2)^\frac{\alpha+1}{2}} - 1\right) \\ = \left( 1 - \frac{1}{\lambda^{\alpha+1}} \right) (1+\lambda^2)^{-1} \left(  \frac{1}{\lambda^{\alpha+1}} - \frac{1}{(1+\lambda^2)^\frac{\alpha+1}{2}}\right)^{-1}.
\end{multline*}
This is equivalent to
\begin{equation}
          \frac{1+\lambda^2}{(1+\lambda^2)^\frac{\alpha+1}{2}} \left( \frac{1-(1+\lambda^2)^\frac{\alpha+1}{2}}{\lambda^2}\right) =  1 - \frac{1}{\lambda^{\alpha+1}}.
\end{equation}
In other words, we are looking for a root of the function $g$ defined by
\begin{equation}
    g(\lambda) =  \frac{1+\lambda^2}{(1+\lambda^2)^\frac{\alpha+1}{2}} \left( \frac{1-(1+\lambda^2)^\frac{\alpha+1}{2}}{\lambda^2}\right) -\left( 1 - \frac{1}{\lambda^{\alpha+1}} \right).
\end{equation}
We have for $\alpha > 0$ that
\begin{equation}
    \lim_{\lambda \rightarrow 0} g(\lambda) = +\infty.
\end{equation}
On the other hand,
\begin{equation}
    g(1) = \frac{1-2^\frac{\alpha+1}{2}}{2^\frac{\alpha-1}{2}} <0.
\end{equation}
Therefore, by the intermediate value theorem applied to the continuous map $g$, there exists $\lambda \in (0,1)$ such that $g(\lambda)=0$. Therefore the system \eqref{eq:system a lambda} has a solution $a \in \RR^\ast_-$, $\lambda \in (0,1)$, and for these values of $a$ and $\lambda$ we have that
\begin{equation}
    \der{\tilde{A}^2}{t} = 4\tilde{\bigtriangleup} \left( \frac{a}{\tilde{B}^{\alpha+1}} - \frac{a}{\tilde{C}^{\alpha+1}} \right),
\end{equation}
\begin{equation}
    \der{\tilde{B}^2}{t} = 4\tilde{\bigtriangleup} \left( \frac{1}{\tilde{C}^{\alpha+1}} - \frac{1}{\tilde{A}^{\alpha+1}} \right)
\end{equation}
and
\begin{equation}
    \der{\tilde{C}^2}{t} = 4\tilde{\bigtriangleup} \left( \frac{1}{\tilde{A}^{\alpha+1}} - \frac{1}{\tilde{B}^{\alpha+1}} \right),
\end{equation}
namely 
\begin{equation}
    \der{\tilde{X}}{t} = F(\tilde{X}).
\end{equation}
Since $\tilde{X}(0) = X(0)$, by the Cauchy Lipschitz theorem, we have that $\tilde{X}(t) = X(t)$ for all $t\in[0,T)$. Now we prove that our solution is a collapse.
By construction our configuration of point-vortices is a self-similar orthogonal triangle, so we have that $\bigtriangleup = \frac{1}{2} AB$. Therefore,
\begin{equation}
    \der{A^2}{t} = 2a \lambda A^2 \ \left(\frac{1}{\lambda^{\alpha+1}A^{\alpha+1}} - \frac{1}{(1+\lambda^2)^\frac{\alpha+1}{2}A^{\alpha+1}}  \right)
\end{equation}
and thus, since $a<0$, there exists a constant $K>0$ depending only on $a$, $\lambda$ and $\alpha$ such that
\begin{equation}
    \der{A}{t} = -\frac{K}{\alpha+1} \frac{1}{A^\alpha}.
\end{equation}
We integrate to get
\begin{equation}
    A^{\alpha+1}(t)-1 = - Kt,
\end{equation}
and thus
\begin{equation}
    A(t) = (1-Kt)^{\frac{1}{1+\alpha}}.
\end{equation}
So there is a collapse at the time $T = \frac{1}{K}$ which depends only on $a$, $\lambda$ and $\alpha$. Since we have a collapse, we can come back to Section \ref{section:appendix_necessary} to recall that the quantity $L$ is conserved and must vanish~\eqref{eq:CN L(X) self similar collapse}. Applying this to our situation gives
\begin{equation}
    a (1+\lambda^2) + a \lambda^2 + 1 = 0
\end{equation}
and thus
\begin{equation}
    a = - \frac{1}{1+2\lambda^2}.
\end{equation}
This formula, somehow much simpler that the one we found, gives us directly that $a \in (-1,0)$. Therefore, the non neutral clusters hypothesis \eqref{eq:no null partial sum} is satisfied and the center of vorticity is preserved during the motion. This concludes that our solution is a self-similar collapse according to relations \eqref{hyp:self similar def f and theta}, where we just need to translate the points so that the center of vorticity is placed at 0. 

We proved the existence of a self-similar collapse. We also observed above that the trajectories are $\cC^{0,\frac{1}{\alpha+1}}([0,T])$ and not better. So the $1/(\alpha+1)$-Hölder regularity stated in Theorem~\ref{thrm:plane} is optimal. Concerning the optimality of the $1/2$-Hölder regularity for the Euler point-vortex dynamics in bounded domains $\Omega$, it is a consequence of the results established in~\cite{Grotto_Pappalettera_2020}.\vspace{1cm}

%%%%%%%%%%%%%%%%%%%%%%%%%%%%%%%%%%%%%%%%%%%%%%%%%%%%%%%%%%%%%%%%%%%
%%%%%%%%%%%%%%%%%%%%%%%%%%%%%%%%%%%%%%%%%%%%%%%%%%%%%%%%%%%%%%%%%%%
%%%%%%%%%%%%%%%%%%%%%%%%%%%%%%%%%%%%%%%%%%%%%%%%%%%%%%%%%%%%%%%%%%%

%%%%%%%%%%%%%%%%%%%%%%%%%%%%%%%%%%%%%%%%%%%%%%%%%%%%%%%%%%%%%%
%%%%%%%%%%%%%%%%%%%%%%%%%%%%%%%%%%%%%%%%%%%%%%%%%%%%%%%%%%%%%%
%%%%%%%%%%%%%%%%%%%%%%%%%%%%%%%%%%%%%%%%%%%%%%%%%%%%%%%%%%%%%%

\vspace{0.5cm}

\noindent{\Large\textbf{Acknowledgments}}\vspace{0.2cm} 

The authors of this article wishes to thank Dragoş Iftimie, PhD advisor of the first author, for his precious help and valuable comments, for his time and confidence during all the realisation of the article. Special thanks for his meticulous rereading of the article and for interesting discussions on point-vortex problems.\vspace{0.2cm}

The authors acknowledge grant support from the project ``Ondes déterministes et aléatoires'' (ANR-28-CE40-0020) and from the project ``Multiéchelle et Trefftz pour le transport numérique'' (ANR-19-CE46-0004) of the Agence Nationale de la Recherche (France).\vspace{2mm}

Data sharing not applicable to this article as no datasets were generated or analysed during the current study.

There are no Competing Interests.

\bibliographystyle{plain}
\bibliography{bibliography}

\end{document}